\documentclass[11pt,letterpaper]{amsart}

\usepackage{amsmath}
\usepackage{amssymb}
\usepackage{amsfonts}
\usepackage{enumerate}
\usepackage{amsthm}
\usepackage{esint}
\usepackage{bbm}

\newtheorem{theorem}{Theorem}[section]
\newtheorem{lemma}[theorem]{Lemma}
\newtheorem{corollary}[theorem]{Corollary}
\newtheorem{proposition}[theorem]{Proposition}

\theoremstyle{definition}
\newtheorem{definition}[theorem]{Definition}
\newtheorem{remark}[theorem]{Remark}
\newtheorem{example}[theorem]{Example}

\DeclareMathOperator*{\argmin}{arg\,min}

\numberwithin{equation}{section}

\newcommand{\rme}{\mathrm{e}}

\newcommand{\R}{\mathbb{R}}
\newcommand{\N}{\mathbb{N}}

\newcommand{\C}{\mathbb{C}}

\usepackage{mathtools}

\DeclarePairedDelimiter{\abs}{\lvert}{\rvert}
\DeclarePairedDelimiter{\norm}{\lVert}{\rVert}
\DeclareMathOperator{\dist}{dist}
\DeclareMathOperator{\diam}{diam}

\DeclareMathOperator{\bdry}{bdry}
\DeclareMathOperator{\sing}{sing}

\DeclareMathOperator{\divergence}{div}

\usepackage{hyperref}

\title[Mosco convergence and Sobolev inequalities]{Mosco convergence of Sobolev spaces and Sobolev inequalities for nonsmooth domains}

\author[M.~Fornoni]{Matteo Fornoni}
\author[L.~Rondi]{Luca Rondi}

\address[M.~Fornoni]{Dipartimento di Matematica, Universit\`a degli Studi di Pavia, Italy}
\email{matteo.fornoni01@universitadipavia.it}

\address[L.~Rondi]{Dipartimento di Matematica, Universit\`a degli Studi di Pavia, Italy}
\email{luca.rondi@unipv.it}

\begin{document}

\setcounter{section}{0}
\setcounter{secnumdepth}{2}

\begin{abstract}
We find extremely general classes of nonsmooth open sets which guarantee Mosco convergence for corresponding Sobolev spaces and the validity of Sobolev inequalities with a uniform constant. An important feature of our results is that the conditions we impose on the open sets for Mosco convergence and for the Sobolev inequalities are of the same nature, therefore it is easy to check when both are satisfied.

Our analysis is motivated, in particular, by the study of the stability of the direct acoustic scattering problem with respect to the scatterer, which we also discuss.

Concerning Mosco convergence in dimension 3 or higher, our result extends all those previously known in the literature.
Concerning Sobolev inequalities, our approach seems to be new and considerably simplifies the conditions previously required for the stability of acoustic direct scattering problems.

\medskip


\medskip

\end{abstract}

\maketitle

\section{Introduction}

The purpose of this paper is two-fold. We consider highly nonsmooth open sets in $\R^N$, $N\geq 2$. First, we investigate 
Mosco convergence of Sobolev spaces defined on these sets. Second, we study the validity of the Sobolev inequality. In particular, we find general classes of sets for which Mosco convergence holds or the Sobolev inequality holds with a constant which is independent on the set but depends only on the class.
Both these results are of independent interest. However, another interesting feature is that the conditions we impose on the sets for the two issues at hand are of the same nature, making much easier to check whether for some class of sets both properties are true. This is relevant since, following \cite{rondi:ac}, it is known that this is needed to guarantee stability of solutions to acoustic scattering problems as the sound-hard scatterer varies. Instead,  previously in \cite{rondi:ac} or in \cite{rondi:scatt}, besides those to ensure Mosco convergence, further and rather involved conditions were required to have uniform Sobolev inequalities.
Indeed, as an application, we later prove a stability result for this acoustic scattering problem that greatly generalizes the ones of \cite{rondi:ac,rondi:scatt}, since we allow a much larger class of scatterers, with conditions much easier to be verified, and we also allow the coefficients of the corresponding Helmholtz equation to vary.

About Mosco convergence of Sobolev spaces under variations of their domains,
 let $1<p<+\infty$ and let $D_n$, $n\in\N$, be open sets which are uniformly bounded.
The aim is to prove that, for some open set $D$, we have
Mosco convergence of $W^{1,p}(D_n)$ to $W^{1,p}(D)$. Also Deny-Lions spaces $L^{1,p}$ or $H(\mathrm{curl})$ spaces are of interest, but these cases can actually be often reduced to the analysis done for Sobolev spaces, see Remark~\ref{otherspacesrem}, therefore we concentrate our attention mostly to the Sobolev case.

The importance of studying this Mosco convergence is due to the fact that it has a well known strict relationship to stability of solutions to Neumann problems, still under variations of the domains, see for instance the characterisation in Proposition~\ref{equiv:mosconeumann}. In turn, the stability of solutions to Neumann problems has important applications in several fields, from fracture mechanics, see for instance \cite{dalmaso-toader}, to stability of direct scattering problems with varying scatterers, see \cite{rondi:ac,rondi:scatt} for the acoustic case and \cite{rondi:em} for the electromagnetic case where $H(\mathrm{curl})$ spaces come into play.

Typically, one assumes some kind of convergence of $D_n$ to $D$ as $n\to\infty$, usually convergence in the Hausdorff complementary topology.
The characterisation of Proposition~\ref{equiv:mosconeumann} implies that
a necessary condition for Mosco convergence is that, as $n\to\infty$, $| D_n | \to |D|$, actually that
$| D_n \Delta D|\to 0$, see Corollary~\ref{convmosco:diffsimm}. Another assumption that seems to be necessary is that the complement of $D_n$ has a uniformly bounded number of connected components, since otherwise we might end up in counterexamples such as the famous Neumann sieve, see \cite{murat}.
In dimension $N=2$, these conditions are indeed sufficient for $1<p\leq 2$, a result first shown in \cite{bucur:varchon0,bucur:varchon} for $p=2$ and then in \cite{dalmaso:nonl} for $1<p\leq 2$. In \cite{dalmaso:nonl} the result was also extended to Deny-Lions spaces $L^{1,p}$, under the same assumptions.
We mention that an important preliminary result in this direction was proved in \cite{chambolle:doveri}, where a uniform bound on the $\mathcal{H}^{1}$ measure of $\partial D_n$ was also required. These $2$-dimensional results made important use of complex analytic techniques, thus 
for $N\geq 3$ (and for $H(\mathrm{curl})$ spaces) it is extremely challenging to reach such a general result. Hence, the aim has been to find as general conditions as possible to ensure Mosco convergence. The breakthrough was in \cite{giacomini}, where $\partial D_n$ belongs to the class $\mathcal{G}(r,L,C)$, see Definition~\ref{classe:grlc},  given by suitable Lipschitz hypersurfaces, and Mosco convergence was proved for Sobolev spaces with $1<p\leq 2$.
In \cite{rondi:ac} the same result was proved by using still Lipschitz hypersurfaces, with slightly more regular boundary, corresponding to the class $\mathcal{MR}(r,L)$ of Definition~\ref{mrclassdef}, that were allowed to intersect each other, in a controlled way, along their boundaries. This result was extended to $H(\mathrm{curl})$ spaces in \cite{rondi:em}. 
We note that a crucial feature in the result of \cite{rondi:ac} was the fact that for $\mathcal{MR}(r,L)$ hypersurfaces converging in the Hausdorff distance we have convergence of their boundaries as well. Such a property is unfortunately not satisfied by $\mathcal{G}(r,L,C)$, see Example~\ref{pacman}, and this makes gluing hypersurfaces in this class not straightforward.

The first aim of the present work is to construct an extremely general class, that contains both the one of \cite{giacomini} and the one of \cite{rondi:ac}, for which Mosco convergence holds, with $1<p\leq 2$. This class extends all previously known results, it is optimal in many respects and its
generality makes it fit for most applications. Let us mention that there are counterexamples, in \cite{dalmaso:nonl} for $N=2$ and in \cite{giacomini} for $N\geq 3$, for the case $p>2$ which seems to be still quite open.

The main difficulty in establishing this new class was in understanding better the properties of hypersurfaces beloging to the class $\mathcal{G}(r,L,C)$ introduced in \cite{giacomini}, which we discuss in the Appendix, and in finding a way to properly glue them together to construct more complicated structures. In Definition~\ref{classe:finale1}, we construct the class $\mathcal{FR}(r,L,M_0)$ given by Lipschitz hypersurfaces which are given by the union of at most $M_0$ hypersurfaces of type $\mathcal{MR}(r,L)$. This new class, which clearly extends the class $\mathcal{MR}(r,L)$, also generalizes classes of the type $\mathcal{G}(r,L,C)$, see Theorem~\ref{inclusione:giac}. Instead of the boundary, for Lipschitz hypersurfaces in $\mathcal{FR}(r,L,M_0)$, the important set is what we call the \emph{singular set} which is the union of the boundaries of the $\mathcal{MR}(r,L)$ components of the hypersurfaces. It turns out that, contrary to the boundary, the singular set is stable under convergence in the Hausdorff distance and this allows us to glue together hypersurfaces from $\mathcal{FR}(r,L,M_0)$, taking care of doing that along their singular sets and not just their boundary as was done in \cite{rondi:ac} for $\mathcal{MR}(r,L)$ hypersurfaces. We obtain a class, $\mathcal{FR}(r,L,M_0,\omega)$ of Definition~\ref{classe:finale2}, which is thus more general than all previous examples. Still, a Mosco convergence result for $1<p\leq 2$ holds, see Theorem~\ref{teor:finale}, the first main result of the paper. We note that our work here is pretty much of a geometrical nature. In fact, once the right geometrical conditions are found, the Mosco result follows straightforwardly from the arguments of \cite{rondi:ac}, which in turn were strongly based on those of \cite{giacomini}.

About Sobolev inequalities, this is a classical problem which has been widely studied and several quite general sufficient conditions have been investigated, see for instance \cite{adams,M} and the references therein. Here we are dealing with a particular case, since the complexity of the sets of $\mathcal{FR}(r,L,M_0,\omega)$ constructed in Section~\ref{Moscosec} makes it hard for them to fit into the well established cases. Moreover, we wish to control the constant of the Sobolev inequality.
The main difficulty lies in the intersection of the hypersurfaces which, although controlled by the function $\omega$, could be very general. If $\omega$ goes to zero fast enough we can create very narrow regions between two different hypersurfaces and estimates there become critical. Thus, in order to have enough control on these in-between regions, we need to make some assumptions on $\omega$. In particular, we assume that, for $s>0$, $\omega(s)=as$ for some positive constant $a$, that is, $\omega$ is linear. This essentially excludes cusp-like situations which are difficult to handle. Another restriction we need to make is that the intersection cannot happen along the singular sets of the different hypersurfaces, but only along their boundaries. In fact, we have a sufficient control on the boundaries of hypersurfaces in $\mathcal{FR}(r,L,M_0)$, whereas very little can be said about their singular sets. However, provided these two restrictions are enforced, we are able to prove that the Sobolev inequality holds  with a constant independent on our set, see Theorem~\ref{sobolevemb:linear}, the second main result of the paper. In our opinion also the method of proof can be of interest, because it employs the so called Friedrichs inequality, in its general version due to Maz'ya, see \cite{M0} and \cite[Corollary on page~319]{M}, which has been further generalised in \cite{rondi:friedrichs}. Indeed, our assumptions on the boundary allow to control the traces of Sobolev functions in suitable integral norms on the boundary. Using this estimate in the Friedrichs inequality leads to an estimate on suitable integral norms of the Sobolev functions inside the domain, thus establishing the Sobolev inequality. In this respect, Proposition~\ref{friedrichs}, which relates the validity of the Sobolev inequality to that of an optimal trace inequality, might be useful in other applications.

The plan of the paper is the following. In Section~\ref{notationsec} we set the notation and discuss general properties of Mosco convergence for Sobolev spaces. In Section~\ref{Moscosec} we construct our classes of hypersurfaces and discuss our Mosco convergence results. In Section~\ref{Sobolevsec}, we discuss conditions on our open sets for the Sobolev inequality to hold. In Section~\ref{section:classscat} we apply our results to the stability of solutions to acoustic scattering problems. Finally, in the Appendix we discuss classes $\mathcal{G}(r,L,C)$ from \cite{giacomini} and we compare them with our classes
$\mathcal{FR}(r,L,M_0)$.

\medskip

\noindent
\textbf{Acknowledgement}\\
Luca Rondi acknowledges support by GNAMPA, INdAM, and by PRIN 2017 
project n.~201758MTR2 funded by Ministero dell'I\-stru\-zio\-ne, dell'Universit\`a
e della Ricerca, Italy.
\
\medskip

\noindent
\textbf{Conflict of interest statement}\\
The authors have no conflict of interest to declare.

\section{General setting and first properties}\label{notationsec}

The integer $N\geq 2$ denotes the space dimension and we note that we drop the dependence of any constant on $N$. 
For any Borel set $E\subset\mathbb{R}^N$, we denote with
$|E|$  its Lebesgue measure, whereas, for any $s\geq 0$, $\mathcal{H}^{s}(E)$ denotes its $s$-dimensional Hausdorff measure.
For any $x\in\mathbb{R}^N$ and any $r>0$, $B_r(x)$ denotes the open ball with center $x$ and radius $r$. Usually $B_r$ stands for $B_r(0)$.
For any $E\subset\mathbb{R}^N$, $E$ not empty, and $y\in\R^N$, we call \emph{distance} of $y$ from $E$ the number $\mathrm{dist}(y,E):=\displaystyle{\inf_{z\in E}|y-z|}$.
For any $E\subset\mathbb{R}^N$, $E$ not empty, we call $\displaystyle{B_r(E):=\bigcup_{x\in E}B_r(x)=\{y\in\R^N\mid\mathrm{dist}(y,E)<r\}}$.

For any $p$, $1\leq p\leq +\infty$, we denote with $p'$ its conjugate exponent, that is $1/p+1/p'=1$. If $p<N$ we call $p^{\ast}$ its Sobolev conjugate exponent, that is $p^{\ast}=\dfrac{pN}{N-p}.$

With $M_{sym}^{N\times N}(\R)$ we denote the space of symmetric real-valued $N\times N$ matrices, whereas $I_N$ denotes the identity $N \times N$ matrix.

By a domain in $\R^N$ we mean an open and connected set.
We say that an open set $\Omega\subset\mathbb{R}^N$
has a \emph{Lipschitz boundary} if for any $x\in\partial\Omega$ there exist a neighbourhood $U_x$  of $x$ and a Lipschitz function $\varphi:\mathbb{R}^{N-1}\to \mathbb{R}$ such that, up to a rigid change of coordinates, we have 
$$\Omega\cap U_x=\{y=(y_1,\ldots,y_{N-1},y_N)\in U_x \mid y_N<\varphi(y_1,\ldots,y_{N-1})\}.$$
We usually denote with $\nu$ the exterior unit normal to $\partial\Omega$.

We say that an open set $\Omega\subset\mathbb{R}^N$
 belongs to the class $\mathcal{A}(r,L,R)$ if $\Omega\subset B_R$ and its boundary is Lipschitz with constants $r$ and $L$ in the following sense: for any $x\in\partial\Omega$ we can choose in the previous definition $U_x=B_r(x)$ and $\varphi$ with Lipschitz constant bounded by $L$.

\begin{remark}\label{oss1} For any bounded open set $\Omega$ with Lipschitz boundary, there exist constants $r$, $L$ and $R$ such that $\Omega\in \mathcal{A}(r,L,R)$.
\end{remark}

\begin{definition} Let $1\leq p<+\infty$.
	A bounded open set $D$ in $\R^N$ satisfies the \emph{Rellich compactness property} for $p$ (RCP$_p$) if the natural immersion of $W^{1,p}(D)$ into $L^p(D)$ is compact.
\end{definition}

\begin{remark}\label{Rellichremark}
If a bounded open set $D$ is Lipschitz, then it satisfies RCP$_p$ for any $1\leq p<+\infty$. Another sufficient condition is a \emph{higher integrability property}. Namely, if
there exist $q>p$ and $C>0$ such that
\begin{equation}\label{higher}
\|v\|_{L^{q}(D)}\leq C\|v\|_{W^{1,p}(D)}\quad\text{for any }v\in W^{1,p}(D),
\end{equation}
then RCP$_p$ holds.
\end{remark}

A function $\omega$ is a \emph{modulus of continuity} if $\omega:(0,+\infty)\to (0,+\infty)$ is a non-decreasing, left-continuous function such that $\displaystyle \lim_{t\to 0^+}\omega(t)=0$. 

Let $\Omega \subset \R^N$ be a fixed bounded open set.
We define $\mathcal{K}(\overline{\Omega}) := \{ K \subset \overline{\Omega}\mid K \text{ not empty compact set} \}$ and $\mathcal{O}(\Omega):=\{D\subset \Omega\mid D\text{ open set}\}$.
We use the following notation. For any $D\in \mathcal{O}(\Omega)$, we call $K(D)=\overline{\Omega}\setminus D\in \mathcal{K}(\overline{\Omega})$.

Given two not empty compact sets $K$ and $\tilde{K}$, the \emph{Hausdorff distance} between $K$ and $\tilde{K}$ is given by
\[ d_H(K,\tilde{K}) := \max \left\{ \sup_{x\in \tilde{K}} \dist(x,K), \sup_{y\in K} \dist(y,\tilde{K}) \right\}. \]
We recall that $(\mathcal{K}(\overline{\Omega}), d_H)$ is a compact metric space, see \cite[Chapter 8]{dalmaso} for instance.

Given $\{ K_n \}_{n\in\N} \subset \mathcal{K}(\overline{\Omega})$ and $K\in\mathcal{K}(\overline{\Omega})$, we say that $K_n$ converges to $K$ in the Hausdorff distance if $d_H(K_n,K) \to 0$ as $n\to \infty$. Analogously, given $\{ D_n \}_{n\in\N} \subset \mathcal{O}(\Omega)$ and $D\in\mathcal{O}(\Omega)$, we say that $D_n$ converges to $D$ in the \emph{Hausdorff complementary topology} if the sequence of compact sets $K_n=K(D_n)$ converges to the compact set $K=K(D)$ in the Hausdorff distance.

\begin{remark}
	The following statements about convergence in the Hausdorff distance are equivalent.
	\begin{enumerate}[1)]
		\item $K_n \to K$ in the Hausdorff distance.
		\item\label{cond2Haus} For any $\varepsilon > 0$ there exists $\bar{n} \in \N$ such that $K\subset B_{\varepsilon}(K_n)$ and $K_n\subset B_{\varepsilon}(K)$ for every $n \ge \bar{n}$.
		\item For any $y\in K$ there exists $\{x_n\}_{n\in\N}$ with $x_n\in K_n$ such that $x_n\to y$ and \\
			for any $\{x_{n_j}\}_{j\in\N}$ with $x_{n_j}\in K_{n_j}$ such that $x_{n_j}\to y$, then $y\in K$.
	\end{enumerate}
	
	By using condition \ref{cond2Haus}), it is also easy to show that if $K_n \to K$ in the Hausdorff distance, then $| K_n \setminus K | \to 0$ as $n\to \infty$.
	
\end{remark}

\begin{remark}
\label{hausd:monotonia}
	Another useful property of the convergence in the Hausdorff distance is the following:
	\[\text{if  }K_n \to K \text{ and } \partial K_n \to \tilde{K} \text{ in the Hausdorff distance, then } \partial K \subset \tilde{K} \subset K. \]
	
	Indeed, the inclusion $\tilde{K}\subset K$ is trivial.  
	Suppose, by contradiction, that $\partial K \not\subset \tilde{K}$. Then, there exists $x\in \partial K$ such that $x \notin \tilde{K}$.
	Since $\partial K_n \to \tilde{K}$, there exist $c>0$ and $\{n_k\}_{k\in\N}$ such that $\dist(x,\partial K_{n_k})>c$ for every $k\in\N$.
	However,
	 $x\in \partial K\subset K$ implies that there exists a sequence $x_n\in K_n$ such that $x_n\to x$, hence $\dist(x_{n_k},\partial K_{n_k}) > c/2$ for any $k\ge \bar{k}$ for some $\bar{k}\in\N$. It follows that $\overline{B_{c/2}(x_{n_k})} \subset K_{n_k}$ for any $k\ge \bar{k}$ and, since $x_{n_k}\to x$, we conclude that $\overline{B_{c/2}(x_{n_k})} \to \overline{B_{c/2}(x)}$ in the Hausdorff distance, thus $\overline{B_{c/2}(x)}\subset K$, which
	 contradicts the fact that $x\in\partial K$.
\end{remark}

\subsection{Mosco convergence for Sobolev spaces}

Let $X$ be a reflexive Banach space and let $\{ A_n\}_{n\in\N}$ be a sequence of closed subspaces of $X$. Denote with $A'$ the set made of the elements of $X$ which are weak limit of a sequence of elements taken from a subsequence of $\{ A_n\}_{n\in\N}$ and with $A''$ the set made of the elements of $X$ which are strong limit of a sequence of elements in $\{ A_n\}_{n\in\N}$. In other words
\[	A' = \{ x\in X \mid x = w\text{-}\lim_{k\to \infty}{x_{n_k}},\, x_{n_k}\in A_{n_k} \}	\] 
\[	A'' = \{ x\in X \mid x = s\text{-}\lim_{n\to \infty}{x_n},\, x_n\in A_n \}.	\] 
Observe that, from the definitions, it is clear that $A'$ and $A''$ are subspaces of $X$ such that $A'' \subset A'$ and that $A''$ is closed in $X$ with respect to the strong topology given by the norm. Then, we can define Mosco convergence.

\begin{definition}
\label{def:mosco}
	Let $X$ be a reflexive Banach space, $\{ A_n\}_{n\in\N}$ be a sequence of closed subspaces of $X$ and  $A'$ and $A''$ be as above. We say that $A_n$ \emph{converges in the sense of Mosco} as $n\to \infty$ to a closed subspace $A\subset X$ if it holds that $A = A' = A''$.
\end{definition}

\begin{remark}
	From the definition and the properties of $A'$ and $A''$ it follows that the condition $A = A' = A''$ is equivalent to the two conditions
	\begin{enumerate}[(M1)]
		\item $[A'\subset A]$ For any $x\in X$, if there exists a subsequence $\{ A_{n_k}\}_{k\in\N}$ and a sequence $\{ x_k \}_{k\in\N}$ with $x_k\in A_{n_k}$ such that $x_k\rightharpoonup x$ weakly as $k\to\infty$, then $x\in A$.
		\item $[A\subset A'']$ For any $x\in A$, there exists a sequence $\{ x_n \}_{n\in\N}$ with $x_n\in A_n$ for every $n\in\N$ such that $x_n\to x$ in $X$ as $n\to\infty$.
	\end{enumerate}
	Therefore, to prove Mosco convergence of $A_n$ to $A$, it is enough to check if the conditions (M1) and (M2) are satisfied.
\end{remark}

Condition (M2) is generally the most difficult to prove, indeed we recall that there are two useful reductions which will be used later on. First, since $A''$ is closed, it is enough to prove (M2) for a subset $S$ dense in $A$. Second, it is enough to prove (M2) for subsequences, that is, to prove that for any $x\in A$ and for any subsequence $\{A_{n_j}\}_{j\in \N}$ there exist a further subsequence $\{A_{n_{j_k}}\}_{k\in \N}$ and elements $x_k \in A_{n_{j_k}}$ such that $x_{k} \to x$ in $X$.

Mosco convergence of Sobolev spaces on varying open sets is defined as follows. Let $N\ge 2$ and $1<p<+\infty$ be fixed. Let $\Omega \subset \R^N$ be a fixed bounded open set.
For any $D\in \mathcal{O}(\Omega)$, 
we have an isometric embedding of $W^{1,p}(D)$ in $L^p(\Omega; \R^{N+1})$ given as follows. For any $u\in W^{1,p}(D)$ we associate $(u,\nabla u)\in L^p(\Omega; \R^{N+1})$ where $u$ and $\nabla u$ are extended to $0$ outside $D$. In this way it is possible to consider $W^{1,p}(D)$ as a closed subspace of $L^p(\Omega; \R^{N+1})$. Then, we can state the following definition.

\begin{definition}\label{Moscodefsob}
	Let $\Omega \subset\R^N$ be a bounded open set. Let $\{ D_n \}_{n\in\N} \subset \mathcal{O}(\Omega)$ and $D\in\mathcal{O}(\Omega)$ and let $1<p<+\infty$. We say that $W^{1,p}(D_n)$ \emph{converges in the sense of Mosco} to $W^{1,p}(D)$ as $n\to \infty$ if, interpreted as closed subspaces of $L^p(\Omega; \R^{N+1})$, such spaces converge in the sense of Definition \ref{def:mosco}.
\end{definition}  

One of the purposes of this work is to find very general conditions on the sequence $\{ D_n \}_{n\in\N} \subset \mathcal{O}(\Omega)$ such that the spaces $W^{1,p}(D_n)$ converge in the sense of Mosco to $W^{1,p}(D)$ for some $D\in \mathcal{O}(\Omega)$. Typically, one assumes  that the sequence $\{D_n\}$ converges in some sense to $D$. For example, one assumes that the sets $D_n$ converge to $D$ in the Hausdorff  complementary topology.

\begin{remark}\label{setminus} Sometimes $D$ is described as $\Omega\setminus K$ for some $K\in\mathcal{K}(\overline{\Omega})$ and properties of $K$ are considered instead of those of $D$. Here we recap how to go from this latter formulation to the one we use in this paper. We have that $K(D)=K\cup\partial\Omega$,
$\partial D\cap \Omega=\partial K\cap \Omega$, $\Omega\setminus \partial K=\Omega\setminus \partial D$,
$\partial (\Omega\backslash \partial K)=\partial K\cup \partial \Omega= K(\Omega\setminus \partial K)$ and
 the following holds.

Let $\{ K_n \}_{n\in\N} \subset \mathcal{K}(\overline{\Omega})$, $K\in\mathcal{K}(\overline{\Omega})$ and $\tilde{K}\in\mathcal{K}(\overline{\Omega})$. We call $D_n=\Omega\setminus K_n$ and $D=\Omega\setminus K$. Then, if $K_n\to K$ in the Hausdorff distance, we conclude that $D_n$ converges to $D$ in the Hausdorff complementary topology. Consequently, if
$\partial K_n\to \tilde{K}$ in the Hausdorff distance, then $\Omega \setminus \partial K_n$ converges to $\Omega\setminus \tilde{K}$ in the Hausdorff complementary topology and $\partial (\Omega\backslash \partial K_n)$ converges  to $\tilde{K}\cup \partial\Omega$
 in the Hausdorff distance.
\end{remark}

\begin{remark}\label{higherMosco}
Sometimes, it is useful that in condition (M1) we have convergence of $u_{n_k}$ to $u$, as $k\to\infty$, strongly in $L^p(\Omega)$. To this aim the following result, which can be easily proved by using a similar argument to the one of \cite[Proposition~2.9]{rondi:ac} for $p=2$, is of interest. 

Let $1\leq p<+\infty$ and let $\{ D_n \}_{n\in\N} \subset \mathcal{O}(\Omega)$ and $D\in \mathcal{O}(\Omega)$ for some arbitrary bounded open set $\Omega$. Assume that $D_n$ converges to $D$ in the Hausdorff complementary topology and that there exist
$q>p$ and $C>0$ such that \eqref{higher} holds uniformly with $D$ replaced by $D_n$ for any $n\in\N$. Let $u_n\in W^{1,p}(D_n)$ and $u\in W^{1,p}(D)$ be such that
$(u_n,\nabla u_n)$ converges to $(u,\nabla u)$ weakly in $L^p(\Omega;\R^{N+1})$, with the usual extensions to $0$. Then $u_n$ converges to $u$ strongly in $L^p(\Omega)$. 
\end{remark}

We now state the first result about Mosco convergence of Sobolev spaces $W^{1,p}$, $1<p<+\infty$. Of particular importance are points \textit{iv}) and \textit{vi}) that extend Proposition~2.2 in \cite{rondi:ac}, which dealt with the case $p=2$ only. We would like to stress that this result is very useful since one can prove Mosco convergence of Sobolev spaces like $W^{1,p}(\Omega \setminus \partial D_n)$ and then use Proposition~\ref{mosco:m1} \textit{iv}) and  \textit{vi}) to recover also convergence for those like $W^{1,p}(D_n)$. Since, with a little amount of regularity, $\partial D_n$ are roughly speaking hypersurfaces in $\R^N$, in the next section we will focus on this kind of closed sets.

\begin{proposition}
\label{mosco:m1}
Let $\Omega \subset\R^N$ be a bounded open set. Let $\{ D_n \}_{n\in\N} \subset \mathcal{O}(\Omega)$ and $D\in\mathcal{O}(\Omega)$ and let $1<p<+\infty$.
	Denote $A_n = W^{1,p}(D_n)$, $n\in\N$, and $A = W^{1,p}(D)$. Then we have the following results.
	\begin{enumerate}[i\textnormal{)}]
	\item	If $D_n\to D$ in the Hausdorff complementary topology and $| D_n \setminus D | \to 0$ as $n\to \infty$, then condition \textnormal{(M1)} of Mosco convergence holds, that is, $A'\subset A$.
	\item If $D_n\subset D$ for every $n\in\N$ and $| D \setminus D_n | \to 0$ as $n\to \infty$, then condition \textnormal{(M2)} of Mosco convergence holds, that is, $A\subset A''$.
	\item If $D_n\to D$ in the Hausdorff complementary topology and $D_n \subset D$ for every $n \in \N$, then $A_n$ converges to $A$ as $n\to \infty$ in the sense of Mosco.
	\item If $D_n\to D$ in the Hausdorff complementary topology, $\partial D_n\cup \partial\Omega \to \tilde{K}$, for some $\tilde{K}\in\mathcal{K}(\overline{\Omega})$, in the Hausdorff distance and $W^{1,p}(\Omega \setminus \partial D_n)$ and $W^{1,p}(\Omega \setminus \tilde{K})$ satisfy condition \textnormal{(M1)}, then condition \textnormal{(M1)} of Mosco convergence holds, that is, $A'\subset A$.
	\item If $D_n\to D$ in the Hausdorff complementary topology,
$\partial D_n\cup \partial\Omega \to \tilde{K}$, for some $\tilde{K}\in\mathcal{K}(\overline{\Omega})$, in the Hausdorff distance and $|\tilde{K}\cap\Omega|=0$, then condition \textnormal{(M1)} of Mosco convergence holds, that is, $A'\subset A$.
	 \item If $D_n\to D$ in the Hausdorff complementary topology, 
$\partial D_n\cup \partial\Omega \to \tilde{K}$, for some $\tilde{K}\in\mathcal{K}(\overline{\Omega})$, in the Hausdorff distance and $W^{1,p}(\Omega \setminus \partial D_n)$ and $W^{1,p}(\Omega \setminus \tilde{K})$ satisfy condition \textnormal{(M2)}, then condition \textnormal{(M2)} of Mosco convergence holds, that is, $A\subset A''$.
\end{enumerate}

\end{proposition}

\begin{proof} For points \textit{i}), \textit{ii}) and \textit{iii}), see Lemma~2.3, Lemma 2.4 and Corollary~2.6 in \cite{rondi:ac} where the same results are proved for $p=2$.

For point \textit{iv}), we note that, as in Remark~\ref{hausd:monotonia}, $\partial D\subset \partial K(D)\cup \partial\Omega \subset \tilde{K}\subset K(D)$.
Then we can use the argument of the first part of the proof of Proposition~2.2. in \cite{rondi:ac}, which is done for $p=2$.
For point \textit{v}), we just use point \textit{iv}) and \textit{i}).

For point \textit{vi}), one can repeat the argument of the second part of the proof of Lemma~4.2 in \cite{rondi:em}.
\end{proof}

\begin{remark}\label{otherspacesrem} An important remark is the following. We can consider Mosco convergence not only for Sobolev spaces $W^{1,p}$, $1<p<+\infty$, but also for Deny-Lions spaces $L^{1,p}$, $1<p<+\infty$. We recall that, for any $D\in \mathcal{O}(\Omega)$,
$L^{1,p}(D)=\{u\in L^1_{loc}(D)\mid \nabla u\in L^p(\Omega)\}$, which is a Banach space if
we identify two elements of $L^{1,p}(D)$ having the same gradient and if we endow it with the norm $\|u\|_{L^{1,p}(D)}=\|\nabla u\|_{L^p(D)}$. Moreover, $L^{1,p}(D)$ can be identified with a closed subspace of $L^p(\Omega;\R^N)$ by associating to $u\in L^{1,p}(D)$ the function $\nabla u$, which is extended to $0$ outside $D$.

If $N=3$, $H(\mathrm{curl})(D)=\{\mathbf{u} \in L^2(D;\R^3)\mid \nabla\times \mathbf{u} \in L^2(D;\R^3)\}$ is a Banach space endowed with the norm
$\| \mathbf{u} \|^2_{H(\mathrm{curl})(D)}=\| \mathbf{u} \|^2_{L^2(D)}+\|\nabla\times \mathbf{u} \|^2_{L^2(D)}$. It can be seen as a closed subspace of $L^2(\Omega;\R^6)$ by associating to $\mathbf{u} \in H(\mathrm{curl})(D)$ the pair $(\mathbf{u},\nabla \times \mathbf{u})$, which are both extended to $0$ outside $D$.

We can define Mosco convergence of $L^{1,p}$ spaces or $H(\mathrm{curl})$ spaces as done for Sobolev spaces in 
Definition~\ref{Moscodefsob}.

In this respect we have the following two observations. First, all results of Proposition~\ref{mosco:m1} extend naturally to $L^{1,p}$, $1<p<+\infty$, and to $H(\mathrm{curl})$ spaces. Moreover, we have the following. Let $\Omega \subset\R^N$ be a bounded open set and let $\{ D_n \}_{n\in\N} \subset \mathcal{O}(\Omega)$ and $D\in\mathcal{O}(\Omega)$, then:

\begin{enumerate}[1)]
\item Let $1<p<+\infty$. If 
	$W^{1,p}(D_n)$, $n\in\N$, and $W^{1,p}(D)$
	satisfy condition (M2), then $L^{1,p}(D_n)$, $n\in\N$, and $L ^{1,p}(D)$ satisfy condition (M2) as well. In fact, it is enough to observe that, by a truncation argument, $W^{1,p}(D)$ is a dense subspace of  $L ^{1,p}(D)$.
\item If 
	$W^{1,2}(D_n)$, $n\in\N$, and $W^{1,2}(D)$
	satisfy condition (M2) and $W^{1,2}(D;\R^3)$ is dense in $H(\mathrm{curl})(D)$,
then $H(\mathrm{curl})(D_n)$, $n\in\N$, and $H(\mathrm{curl})(D)$ satisfy condition (M2) as well.
\end{enumerate}
\end{remark}

We would like to recall that Mosco convergence for Sobolev spaces when $p=2$ has been widely studied thanks to the close relationship between convergence in the sense of Mosco of the spaces $W^{1,2}$ and the stability property of weak solutions to Neumann problems for elliptic equations under variations of the domains, see \cite{chambolle:doveri}, \cite{bucur:varchon0}, \cite{bucur:varchon}, in particular \cite[Proposition~3.2 and Corollary~3.3]{bucur:varchon}. For the sake of completeness, we recall this characterisation by stating and proving a similar one that holds for any $p$, $1<p<+\infty$.

Let $1<p<+\infty$. For any $D \in \mathcal{O}(\Omega)$, any $f\in L^{p'}(D)$ and any $F\in L^{p'}(D;\R^N)$,  let $u=u(D,f,F)$ be the solution to the variational problem
  	\begin{multline}
	\label{variational:probgen}
		\min_{v \in W^{1,p}(D)} \mathcal{J}(v) \\= \min_{v \in W^{1,p}(D)}  \left( \frac{1}{p} \int_{D} \left(\abs{\nabla v}^p + \abs{v}^p\right) - \int_{\Omega} f v   - \int_{\Omega} F \cdot \nabla v \right).
	\end{multline}
	Such a problem has a unique solution $u\in W^{1,p}(D)$, which is characterised by solving, $\text{for any }\varphi \in W^{1,p}(D)$,  
	\begin{equation}
	\label{varprob:weakform} 
		\int_{D} \left(\abs{\nabla u}^{p-2} \nabla u \cdot \nabla \varphi + \abs{u}^{p-2} u \varphi \right)= \int_{D} \left(f \varphi + F \cdot \nabla \varphi\right).
	\end{equation}
	In other words, $u \in W^{1,p}(D)$ is the unique weak solution of the homogeneous Neumann problem
	$$
			 - \divergence(\abs{\nabla u}^{p-2} \nabla u) + \abs{u}^{p-2} u = f - \divergence(F) \quad \text{in }D.
	$$

We have the following result.

\begin{theorem}
\label{equiv:mosconeumann}
	Let $\{ D_n \}_{n\in\N} \subset \mathcal{O}(\Omega)$ and $D\in\mathcal{O}(\Omega)$. Let $1<p<+\infty$.
	Then the following holds.
	\begin{enumerate}[a\textnormal{)}]
		\item If $W^{1,p}(D_n)$ converges in the sense of Mosco to $W^{1,p}(D)$ as $n\to \infty$, then for any $f\in L^{p'}(\Omega)$ and any $F\in L^{p'}(\Omega;\R^N)$ one has that $u_n = u(D_n,f|_{D_n},F|_{D_n})$, solution to \eqref{variational:probgen} with $D_n$ in place of $D$, converges in $L^p(\Omega; \R^{N+1})$ to $u=u(D,f|_D,F|_D)$, solution to \eqref{variational:probgen}, namely $u_n\to u$ in $L^p(\Omega)$ and $\nabla u_n \to \nabla u$ in $L^p(\Omega;\R^N)$ \textnormal{(}with the usual extensions to $0$\textnormal{)}.
		\item Viceversa, if $D_n \to D$ in the Hausdorff complementary topology and if for any $f\in L^{p'}(\Omega)$ one has that $u_n = u(D_n,f|_{D_n},0)$ converges to $u=u(D,f|_D,0)$ in $L^p(\Omega; \R^{N+1})$, then  $W^{1,p}(D_n)$ converges to $W^{1,p}(D)$ in the sense of Mosco as $n\to \infty$.
	\end{enumerate}
\end{theorem}

\begin{proof}
	For any $f\in L^{p'}(\Omega)$, any $F\in L^{p'}(\Omega;\R^N)$ and any $n\in\N$, we call
	\[ \mathcal{J}_n(v) = \frac{1}{p} \int_{D_n}\left( \abs{\nabla v}^p + \abs{v}^p\right) - \int_{D_n} f v-\int_{D_n}F\cdot \nabla v,\]
	and
\[	 \quad u_n = u(D_n,f|_{D_n},F|_{D_n}) =  \argmin_{v \in W^{1,p}(D_n)} \mathcal{J}_n(v), \]
and, correspondingly,
	\[	\mathcal{J}(v) = \frac{1}{p} \int_{D} \left(\abs{\nabla v}^p + \abs{v}^p \right) - \int_{D} f v-\int_{D}F\cdot \nabla v, \] 
	\[ u = u(D,f|_D,F|_D) =  \argmin_{v \in W^{1,p}(D)} \mathcal{J}(v). \]

We begin by proving \textit{a}). By the variational characterisation, it is easy to infer that, for a constant $C$ depending on $\|f\|_{L^{p'}(\Omega)}$, $\|F\|_{L^{p'}(\Omega)}$ and $p$ only, we have that $\norm{u_n}_{W^{1,p}(D_n)} \le C$ for every $n\in\N$, which means that the sequence $\{(u_n,\nabla u_n)\}$ is bounded in $L^p(\Omega; \R^{N+1})$. Hence, there exists $(v,V) \in L^p(\Omega; \R^{N+1})$ and a subsequence $(u_{n_k}, \nabla u_{n_k})$ such that $(u_{n_k}, \nabla u_{n_k})\rightharpoonup (v,V)$ weakly in $L^p(\Omega; \R^{N+1})$. By condition (M1) of Mosco convergence, we immediately deduce that $v\in W^{1,p}(D)$, with $V=\nabla v$ in $D$, and that $v$ and $V$ are $0$ outside $D$. We also obtain that, by lower semicontinuity, $J(v)\leq \liminf_k J_{n_k}(u_{n_k})$.

We now want to show that $v =u = u(D,f|_D,F|_D))$. Indeed, take any $w\in W^{1,p}(D)$, then, by condition (M2) of Mosco convergence, there exists a sequence $w_n \in W^{1,p}(D_n)$ such that $w_n \to w$ strongly in $L^p(\Omega; \R^{N+1})$. Then, by strong continuity of the functionals, we have that $\mathcal{J}_n(w_n) \to \mathcal{J}(w)$ as $n\to \infty$. Therefore, using the minimisation property satisfied by $u_n$, we infer that
$$J(v)\leq \liminf_k J_{n_k}(u_{n_k})\leq \liminf_k J_{n_k}(w_{n_k}) = J(w).$$
By uniqueness of the solution to the minimisation problem \eqref{variational:probgen}, we conclude that 
 $v =u = u(D,f|_D,F|_D)$.

 Since $u_n$ and $u$ are solutions to the minimisation problems, by using the weak formulation \eqref{varprob:weakform}, testing with $\varphi_n = u_n$ and $\varphi = u$ respectively, and exploiting the weak convergence $u_{n_k} \rightharpoonup u$, we obtain that
 $$\lim_k\|(u_{n_k},\nabla u_{n_k})\|_{L^p(\Omega;\R^{N+1})}=\|(u,\nabla u)\|_{L^p(\Omega;\R^{N+1})},$$
 hence $(u_{n_k},\nabla u_{n_k}) \to (u,\nabla u)$ strongly in $L^p(\Omega;\R^{N+1})$ as $k\to\infty$. By Urysohn, we conclude that the whole sequence $(u_{n},\nabla u_{n}) \to (u,\nabla u)$ strongly in $L^p(\Omega;\R^{N+1})$ as $n\to\infty$.
 		
			Now we proceed to prove \textit{b}).
		To show condition (M1), we apply Proposition~\ref{mosco:m1}. By hypothesis, we already have that $D_n \to D$ in the Hausdorff complementary topology, then we just have to show that $|D_n \setminus D | \to 0$ as $n\to \infty$. Consider $f \equiv 1 \in L^{p'}(\Omega)$ and the corresponding solutions $u_n=u(K_n,1,0)$ and $u=u(K,1,0)$ to \eqref{variational:probgen}. By uniqueness of the solution, we have that $u \equiv 1$ in $D$ and $u_n \equiv 1$ in $D_n$, for any $n\in\N$. Additionally, by hypothesis, we know that $u_n \to u$ in $L^p(\Omega)$, with the usual estensions to $0$ outside $D_n$ and $D$, respectively. It follows that
		\[	|D_n \Delta D| = \norm{u_n - u}_{L^p(\Omega)}^p \to 0 \quad \text{as } n \to \infty,	\]
hence $|D_n \setminus D | \to 0$ as $n\to \infty$ and condition (M1) is proved.
		
		Next, we recall that it is enough to show condition (M2) only on a dense subset of $W^{1,p}(D)$. Define the resolvent operator $R_D: L^{p'}(D) \to W^{1,p}(D)$ such that $R_D(f) = u(D,f,0)$, solution to \eqref{variational:probgen}. Then, if we can show that $R_D(L^{p'}(D))$ is dense in $W^{1,p}(D)$, condition (M2) easily follows. Indeed, take $u \in R_D(L^{p'}(D))$, then $u=u(D,f,0)$ for a certain $f\in L^{p'}(D)$. Now, extend $f$ to $0$ outside $D$ and consider the solutions $u_n = u(D_n,f|_{D_n},0) \in W^{1,p}(D_n)$ for every $n\in\N$. It follows directly from our hypothesis that $(u_n, \nabla u_n) \to (u, \nabla u)$ in $L^p(\Omega;\R^{N+1})$, which proves (M2). 
		
		To show that $R_D(L^{p'}(D))$ is dense in $W^{1,p}(D)$, take $u\in W^{1,p}(D)$. Then, $u=u(D,f,F)$ where
		$f = \abs{u}^{p-2} u \in L^{p'}(D)$ and $F = \abs{\nabla u}^{p-2} \nabla u \in L^{p'}(D;\R^N)$. We call $F^*=(f,F)\in (W^{1,p}(D))^*$ the corresponding element of the dual. Since $W^{1,p}(D)$ is dense in $L^p(D)$, 
		$L^{p'}(D)$ is dense in $(W^{1,p}(D))^*$, thus there exists $\{f_n \}_{n\in\N} \subset L^{p'}(D)$ such that $(f_n,0) \to F^*$ in $(W^{1,p}(D))^*$. As a consequence, $u_n = u(D,f_n,0) \in R_D(L^{p'}(D))$ for every $n\in\N$ and it is not difficult to prove that $u_n \to u$ in $W^{1,p}(D)$.\end{proof}

\begin{corollary}
\label{convmosco:diffsimm}
	Let $\{ D_n \}_{n\in\N} \subset \mathcal{O}(\Omega)$ and $D\in\mathcal{O}(\Omega)$. Let $1<p<+\infty$. 
	If $W^{1,p}(D_n)$ converges to $W^{1,p}(D)$ in the sense of Mosco as $n\to \infty$, then one has that $| D_n \Delta D|\to 0$ as $n\to \infty$ and, in particular, $| D_n | \to |D|$ as $n\to\infty$.
\end{corollary}	

\begin{proof} By using part \textit{a}) of the previous theorem, one can use the argument with $f \equiv 1$ used in the proof of part \textit{b}) of the same theorem.
\end{proof}

\section{Classes of hypersurfaces and Mosco convergence}\label{Moscosec}

Let us fix $N \ge 2$ and $\Omega \subset \R^N$ an open and bounded set. In the notation that we will use to name the following classes of compact sets, we will suppress the dependence on $\Omega$, since here it is considered fixed, but they will clearly depend on it. We begin with the following definition of Lipschitz hypersurfaces.

\begin{definition}
\label{ipersup:lip}
	We say that $K\in \mathcal{K}(\overline{\Omega})$ is a (\emph{compact}) \emph{Lipschitz hypersurface with constants} $r>0$ \emph{and} $L>0$, with or without boundary, if for any $x\in K$ there exists a bi-Lipschitz map $\Phi_x : B_r(x) \subset \R^N \to \R^N$ such that
	\begin{enumerate}[(a)]
		\item $L^{-1} \abs{z_1 - z_2} \le \abs{ \Phi_x(z_1) - \Phi_x(z_2) } \le L \abs{z_1 - z_2} \quad \text{for any } z_1, z_2 \in B_r(x)$.
		\item $\Phi_x(x) = 0$ and $\Phi_x(K\cap B_r(x)) \subset \pi := \{ y\in\R^N \mid y_N = 0 \}$. 
	\end{enumerate}
	Moreover, for such $K$ we say that $x\in K$ \emph{belongs to the interior of} $K$ if there exists $ 
	\delta >0$ such that $B_{\delta}(0) \cap \pi \subset \Phi_x(K \cap B_r(x))$ and we define:
	\[	\bdry(K) := \{ x\in K \mid x \text{ does not belong to the interior of $K$}\}	\]
\end{definition}

\begin{remark}\label{bdrychar}
	It is easy to prove that there exists $C_1=C_1(r,L,\Omega)>0$ such that for any Lipschitz hypersurface $K$
\begin{equation}\label{n-1bound}
	\mathcal{H}^{N-1}(K) \le C_1 < + \infty,
	\end{equation}
	 and that $\bdry(K)$ is compact. Moreover, we have that $\Phi_x(B_r(x))$ is an open set containing $B_{r/L}$. We also note that if $z\in B_{r/L}$ belongs to the boundary of $\Phi_x(K\cap B_r(x))$ as a subset of $\pi$, then $y=\Phi_x^{-1}(z)$ belongs to $\bdry(K)$.
\end{remark}

\begin{remark}\label{stabab} Let us note that Lipschitz hypersurfaces of Definition~\ref{ipersup:lip} are stable with respect to Hausdorff convergence. In fact,
let $K_n$, $n\in\N$, satisfy conditions (a) and (b) and let $K$ be the limit of $K_n$ in the Hausdorff distance. Let $x\in K$ and let $x_n\in K_n$ be such that $x_n\to x$ as $n\to \infty$.
	Observe that, up to a translation, the maps $\Phi_{x_n}^n$ can be rewritten as $\Phi_{x_n}^n = \tilde{\Phi}_n(\cdot - x_n)$ with $\tilde{\Phi}_n \colon \overline{B_r(0)} \to \R^N$ bi-Lipschitz with constant $L$. Then, by Ascoli-Arzelà theorem, we have that, up to a subsequence, $\tilde{\Phi}_n \to \tilde{\Phi}$ uniformly with $\tilde{\Phi} \colon \overline{B_r(0)} \to \R^N$ bi-Lipschitz with constant $L$. Now define $\Phi_x := \tilde{\Phi}(\cdot - x)$. Clearly $\Phi_x$ satisfies (a) and, since $\tilde{\Phi}_n(0) = 0$ for every $n\in \N$, also $\Phi_x(x)=0$. In order to show (b), we only need to prove that $\Phi_x(K\cap B_r(x)) \subset \pi$. Indeed, let $y\in K\cap B_r(x)$, then, by Hausdorff convergence, there exist $y_n \in K_n$ such that $y_n \to y$. In particular, there exists $\bar{n}\in\N$ such that, for every $n\ge \bar{n}$, $y_n \in B_r(x_n) \cap K_n$, so that $\tilde{\Phi}_n(y_n - x_n)\in \pi$. Now observe that, by uniform convergence, $\tilde{\Phi}_n(y_n - x_n) \to \tilde{\Phi}(y-x)$, therefore $\Phi_x(y) = \tilde{\Phi}(y-x) \in \pi$ and (b) is proved.  
\end{remark}

In order to introduce some regularity on the boundary of a Lipschitz hypersurface $K$, we recall the following definition from  
\cite{rondi:ac}.

\begin{definition}\label{mrclassdef}
	A Lipschitz hypersurface $K\in\mathcal{K}(\overline{\Omega})$ belongs to the \emph{class} $\mathcal{MR}(r,L)$ with parameters $r>0$ and $L>0$, if, in addition to hypotheses (a) and (b) of Definition \ref{ipersup:lip}, it satisfies
	\begin{enumerate}[(a)]\setcounter{enumi}{2}
		\item for any $x\in \bdry(K)$ one has that $\Phi_x (K\cap B_r(x)) = \Phi_x(B_r(x)) \cap \pi^+$, where we set $\pi^+ := \{ y \in \R^N \mid y_N = 0,\ y_{N-1} \ge 0 \}$.
	\end{enumerate}
\end{definition}

It is easy to prove that there exists $C_2=C_2(r,L,\Omega)>0$ such that, for any $K\in \mathcal{MR}(r,L)$,
\begin{equation}\label{n-2bound}
\mathcal{H}^{N-2}(\bdry(K)) \le C_2 < + \infty.
\end{equation}

We generalise the $\mathcal{MR}(r,L)$ class in the following way.
\begin{definition}
\label{classe:finale1}
	A Lipschitz hypersurface $K\in\mathcal{K}(\overline{\Omega})$ belongs to the \emph{class} $\mathcal{FR}(r,L,M_0)$ with parameters $r>0$, $L>0$ and $M_0\in\N$, if, in addition to hypotheses (a) and (b) of Definition \ref{ipersup:lip}, we have that
$K=\bigcup_{j=1}^J H^j$, where $H^j\in \mathcal{MR}(r,L)$ for every $j=1,\ldots,J$ and $1\leq J\le M_0$.

	Moreover, for such $K$ we define the \emph{set of singular points} as:
	\[	\sing(K) := \bigcup_{j=1}^J \bdry(H^j).	\]
\end{definition}

We point out that the previous definition states that $K\in \mathcal{FR}(r,L,M_0)$ is globally a Lipschitz hypersurface as in Definition \ref{ipersup:lip}, but every single piece of the decomposition $K=\bigcup_{j=1}^J H^j$ belongs to the class $\mathcal{MR}(r,L)$. In particular, note that if $x\in H^j$, the map $\Phi_x$ which satisfies hypothesis (a) and (b) globally for $K$ and the one that satisfies hypotheses (a), (b) and (c) for $H^j$ could be different. 

\begin{remark}	An important remark is that the decomposition, or representation, of $K$ in terms of the sets $H^j\in \mathcal{MR}(r,L)$ is not unique. In particular, 
	since $K \in \mathcal{FR}(r,L,M_0)$ is globally a Lipschitz hypersurface, $\bdry(K)$ is well defined as in Definition \ref{ipersup:lip} and depends only on $K$. On the other hand, the set $\sing(K)$ is not intrinsically defined by $K$, but it strongly depends on the particular choice of the decomposition. However, it is always true that $\bdry(K) \subset \sing(K)$ and $\sing(K)$ is compact.
\end{remark}

By \eqref{n-1bound} and \eqref{n-2bound}, we infer that, for any $K\in \mathcal{FR}(r,L,M_0)$, one has
	\[  \mathcal{H}^{N-1}(K) \le C_1 < + \infty \quad \text{and} \quad \mathcal{H}^{N-2}(\sing(K)) \le M_0C_2 < + \infty,	\]
	independently on the representation.

We study the closure, and thus the compactness, of class $\mathcal{FR}(r,L,M_0)$ with respect to convergence in the Hausdorff distance. The interesting result here is that, up to subsequences, we also have convergence of the corresponding sets of singular points.

\begin{proposition}
\label{convh:finale1}
	Let $r>0$, $L>0$ and $M_0\in\N$ and let $\{ K_n\}_{n\in\N} \subset \mathcal{FR}(r,L,M_0)$, with given decompositions.
	 Assume that $K_n \to K$ in the Hausdorff distance as $n\to\infty$, with $K\in\mathcal{K}(\overline{\Omega})$.  
	Then, $K\in \mathcal{FR}(r,L,M_0)$ and, up to a subsequence, one has that, for a suitable decomposition of $K$, $\sing(K_n) \to \sing(K)$ in the Hausdorff distance as $n\to\infty$.
\end{proposition}

\begin{proof}
	Let $\{ K_n\}_{n\in\N} \subset \mathcal{FR}(r,L,M_0)$ be such that $K_n \to K$ in the Hausdorff distance. By Remark~\ref{stabab}, we know that $K$ satisfies conditions (a) and (b) of Definition~\ref{ipersup:lip}.
		
	Since the number of components of the decomposition of $K_n$ is limited by $M_0$ for any $n\in\N$, it is possible to find a subsequence in which the number of components $J$ is the same for every $n\in\N$. Without relabeling, we have that $K_n = \bigcup_{j=1}^J H_n^j$ with $H_n^j\in\mathcal{MR}(r,L)$ for every $j=1,\ldots,J$ and $n\in\N$. Now, by \cite[Lemma~3.6]{rondi:ac}, up to passing to a subsequence, we have that $H_n^j\to H^j\in\mathcal{MR}(r,L)$ and $\bdry(H_n^j)\to \bdry(H^j)$ in the Hausdorff distance for every $j=1,\ldots,J$. Then, it immediately follows that $K=\bigcup_{j=1}^J H^j$, which implies $K\in\mathcal{FR}(r,L,M_0)$, and that $\sing(K_n)\to\sing(K)$ in the Hausdorff distance.
\end{proof}

Following the idea in \cite{rondi:ac}, where hypersurfaces belonging to $\mathcal{MR}(r,L)$ where suitably glued along their boundaries, we consider sets where hypersurfaces belonging to $\mathcal{FR}(r,L,M_0)$ can be analogously glued together, this time along their singular sets rather than along their boundaries.

\begin{definition}
\label{classe:finale2}
	We say that a compact set $K\in\mathcal{K}(\overline{\Omega})$ belongs to the class $\mathcal{FR}(r,L,M_0,\omega)$ with parameters $r>0$, $L>0$, $M_0\in\N$ and $\omega:(0,+\infty)\to(0,+\infty)$ modulus of continuity, if $K$ satisfies
	\begin{enumerate}[(1)]
		\item $K=\bigcup_{i=1}^M K^i$, where $K^i\in \mathcal{FR}(r,L,M_0)$ for every $i=1,\ldots,M$, for some $M\in\N$.
		\item For any $i=1,\ldots,M$ and $x\in K^i$, if $\dist(x,\sing(K^i)) = \delta > 0$, then it follows that $\dist(x,\bigcup_{j\neq i} K^j) \ge \omega(\delta)$.
	\end{enumerate}
	Moreover, we define $\sing(K) := \bigcup_{i=1}^M \sing(K^i)$.
\end{definition}

We remark that, as before, the decomposition, or representation, of $K$ as union of $K^i\in \mathcal{FR}(r,L,M_0)$ is not at all unique, as well as the decomposition of each $K^i$ into its $\mathcal{MR}(r,L)$ components. Therefore $\sing(K)$ is not intrinsically defined and strongly depends on the chosen representation of $K$ in terms of $K^i$ and, in turn, of $K^i$ in terms of its $\mathcal{MR}(r,L)$ components.

We note that, given a representation of $K$, for any $i,j=1,\ldots,M$ with $i\ne j$ one has that $K^i \cap K^j \subset \sing(K^i) \cap \sing(K^j)$, which means that two different hypersurfaces in the decomposition of $K\in\mathcal{FR}(r,L,M_0,\omega)$ are either disjoint or can intersect only along the sets of singular points.

Definition~\ref{classe:finale1} does not guarantee any control on the intersection of the various pieces $H^j$, which means that the set $\sing(K)$ can be difficult to describe in general. This is clearly true also for Definition~\ref{classe:finale2}.  However, a basic property like the following still holds.

\begin{lemma}
\label{raggio:finale}
	Let $r>0$, $L>0$, $M_0\in\N$ and $\omega$ be a modulus of continuity and let $K\in \mathcal{FR}(r,L,M_0,\omega)$. Then for any $x\in K$ and any $\bar{r}>0$ there exists
$\bar{s}>0$, depending only on $r$, $L$, $M_0$, $\omega$ and $\bar{r}$, and $y\in B_{\bar{r}/2}(x)\cap K$ such that $\dist(y,\sing(K)) \ge \bar{s}$.
\end{lemma}

\begin{proof} We begin by proving this property for $K\in \mathcal{FR}(r,L,M_0)$.
Without loss of generality,
up to reordering the sets of its decomposition, we can assume that $x\in H^1$ and that $0<\bar{r}<r/2$. We call $x_0=x$ and $s_0=\bar{r}/4$.
 If $\bdry(H^1)\cap B_{s_0/4}(x_0)=\emptyset$, we call
$x_1=x_0$ and $s_1=s_0/4$. 
Otherwise, there exists $z_1 \in \bdry(H^1) \cap B_{s_0/4}(x_0)$ and it is clear that $B_{s_0/4}(z_1) \subset B_{s_0/2}(x_0)\subset B_{\bar{r}/4}(x)$. 
Now take the bi-Lipschitz map $\Phi_{z_1}^1 \colon B_r(z_1) \to \R^N$ defined on $H^1\in \mathcal{MR}(r,L)$ such that $\Phi_{z_1}^1(H^1 \cap B_r(z_1)) = \Phi_{z_1}^1(B_r(z_1)) \cap \pi^+$. In particular, we have that $B_{s_0/(16L)} \cap \pi^+ \subset \Phi_{z_1}^1(H^1 \cap B_{s_0/16}(z_1))$. We can find $w_1\in B_{s_0/(16L)} \cap \pi^+$ such that $B_{s_0/(64L)}(w_1)\cap \pi \subset B_{s_0/(16L)}\cap \pi^+$. Therefore, by taking $x_1= (\Phi_{z_1}^1)^{-1}(w_1)$ and $s_1 = s_0/(64L^2)$, it follows that $B_{s_1}(x_1) \subset B_{s_0/2}(x_0)$ and $\bdry(H^1) \cap B_{s_1}(x_1) = \emptyset$. As a consequence, assuming, without loss of generality, that $L\geq 1$,
	for $s_1=s_0/(64L^2)$ there exists $x_1\in K$ such that $\dist(x_1,\bdry(H^1)) \ge s_1$ and $B_{s_1}(x_1)\subset B_{s_0/2}(x_0)$.

Then, we apply the exact same reasoning to the set $H^2$ and we can find, for  $s_2=s_1/(64L^2)$, $x_2\in K$ such that
$\dist(x_2,\bdry(H^2)) \ge s_2$ and $B_{s_2}(x_2)\subset B_{s_1/2}(x_1)$, hence we also have  $\dist(x_2,\bdry(H^2)\cup \bdry(H^1)) \ge s_2$.
 	 By repeating the previous argument up to $M_0$, we conclude that the thesis holds with $\bar{s}=\dfrac{\bar{r}}{4(64L^2)^{M_0}}$.
	 
	 For the general case, let $x\in K$, with $K\in \mathcal{FR}(r,L,M_0,\omega)$. Without loss of generality,
up to reordering the sets of its decomposition, we can assume that $x\in K^1$. We have just proved that there exists
$y\in B_{\bar{r}/2}(x)\cap K^1$ such that $\dist(y,\sing(K^1)) \ge \bar{s}$. We conclude that $\dist(y,\bigcup_{j\neq 1} K^j)\geq \omega(\bar{s})$ and the proof is concluded.
\end{proof}

\begin{remark}
\label{upperbound:finale}
	The number $M\in\N$ of components can depend on the particular $K\in\mathcal{FR}(r,L,M_0,\omega)$, however, using Lemma \ref{raggio:finale}, one can find a uniform upper bound $\overline{M}\in\N$, depending only on parameters $r$, $L$, $M_0$, $\omega$ and on $\Omega$, such that $M\le \overline{M}$ for any $K\in\mathcal{FR}(r,L,M_0,\omega)$.  Hence,
for any $K\in \mathcal{FR}(r,L,M_0,\omega)$ one has
	\[	\mathcal{H}^{N-1}(K) \le \overline{M}C_1 < + \infty \quad \text{and} \quad  \mathcal{H}^{N-2}(\sing(K)) \le \overline{M}M_0C_2 < + \infty,	\]
	independently on the representation.
	\end{remark}
	
We conclude by introducing the following two classes
\[ \widetilde{\mathcal{FR}}(r,L,M_0,\omega) := \{ K \in \mathcal{K}(\overline{\Omega}) \mid  \partial K \in \mathcal{FR}(r,L,M_0,\omega) \} \]
and
\[ \mathcal{O}(r,L,M_0,\omega) := \{ D \in \mathcal{O}(\Omega) \mid  \partial D \in \mathcal{FR}(r,L,M_0,\omega) \}. \]
	
The analogous of Proposition~\ref{convh:finale1} still holds for all these new classes.

\begin{proposition}
\label{convh:finale2}
	Let $r>0$, $L>0$, $M_0\in\N$ and $\omega:(0,+\infty)\to(0,+\infty)$ be a modulus of continuity.
	Let $\{ K_n\}_{n\in\N} \subset \mathcal{FR}(r,L,M_0,\omega)$, with given decompositions.
	 Assume that $K_n \to K$ in the Hausdorff distance as $n\to\infty$, with $K\in\mathcal{K}(\overline{\Omega})$.
	Then, $K\in \mathcal{FR}(r,L,M_0,\omega)$ and, up to a subsequence, one has that, for a suitable decomposition of $K$, $\sing(K_n) \to \sing(K)$ in the Hausdorff distance as $n\to\infty$.
	
	Moreover, the class $\widetilde{\mathcal{FR}}(r,L,M_0,\omega)$ is closed, thus compact, with respect to convergence in the Hausdorff distance and the class $\mathcal{O}(r,L,M_0,\omega)$ is also closed, thus compact, in the Hausdorff complementary topology.
\end{proposition}

\begin{proof} For the class $\mathcal{FR}(r,L,M_0,\omega)$, 
one can easily adapt the arguments of the proof of \cite[Lemma 3.8]{rondi:ac}, by using Propostition~\ref{convh:finale1} and substituting $\bdry(\cdot)$ with $\sing(\cdot)$. 

	Let $\{K_n\}_{n\in\N} \subset \widetilde{\mathcal{FR}}(r,L,M_0,\omega)$ be such that $K_n \to K$ in the Hausdorff distance.	
	We already know that, up to a subsequence, $\partial K_n \to \tilde{K} \in \mathcal{FR}(r,L,M_0,\omega)$ in the Hausdorff distance and, moreover, that $\partial K \subset \tilde{K} \subset K$ by Remark \ref{hausd:monotonia}. Therefore, we only need to show that $\tilde{K}=\partial K$.

By contradiction, suppose there exists $x\in \tilde{K} \setminus \partial K$. Then clearly $x$ is a point belonging to the interior of $K$, namely there exists $d>0$ such that $B_d(x) \subset K$. Moreover, by the convergence in the Hausdorff distance of $\partial K_n \to \tilde{K}$, there exists a sequence of points $x_n\in \partial K_n$ such that $x_n\to x$. Take $\bar{r}=d/4$ and let $y_n\in B_{\bar{r}/2}(x_n)\cap \partial K_n$ such that 
$\dist(y_n,\sing(\partial K_n)) \ge \bar{s}$
as in Lemma~\ref{raggio:finale}. We can find a radius $\bar{s}_1>0$, depending on $\bar{s}$ and $L$ only, and $z_n$ such that
$B_{\bar{s}_1}(z_n)\subset B_{\bar{s}}(y_n)\setminus K_n$. Up to a subsequence, $z_n\to z\in \overline{B_{d/2}(x)}$, therefore $z\in K$. Hence, there exists $w_n\in K_n$ such that $w_n\to z$. But, for $n$ large enough, $w_n\in B_{\bar{s}_1}(z_n)$ and we obtain a contradiction.

Finally, let $\{D_n\}_{n\in\N} \subset \mathcal{O}(r,L,M_0,\omega)$ be such that $D_n \to D$ in the Hausdorff complementary topology, that is, $\overline{\Omega}\setminus D_n\to \overline{\Omega}\setminus D$ in the Hausdorff distance. Assuming $\Omega\subset B_R$, for some $R>0$, we have that $D_n \to D$ in the Hausdorff complementary topology if and only if
$K_n=\overline{B_{R+1}}\setminus D_n$ converges to $K=\overline{B_{R+1}}\setminus D$. But, for any $n\in\N$, $K_n$ belongs to the class
$\widetilde{\mathcal{FR}}(r,L,M_0,\omega)$ when $\Omega$ is replaced by $B_{R+1}$. We conclude that $K$ belongs to the same class and thus $D\in \mathcal{O}(r,L,M_0,\omega)$.\end{proof}

Finally, we state and prove our Mosco convergence result.

\begin{theorem}
\label{teor:finale}
	Let $\Omega \subset \R^N$ be an open bounded set. Let $r>0$, $L>0$, $M_0\in\N$ and $\omega:(0,+\infty)\to(0,+\infty)$ be a modulus of continuity. Let $\{ D_n\}_{n\in\N} \subset \mathcal{O}(r,L,M_0,\omega)$ and $D\in \mathcal{O}(\Omega)$ be such that $D_n \to D$ in the Hausdorff complementary topology. Let $1 < p \le 2$.  \par
	Then $W^{1,p}(D_n)$ converges to $W^{1,p}(D)$ in the sense of Mosco as $n\to\infty$.
\end{theorem}

\begin{proof} We begin with the following case. Let $K_n\in \mathcal{FR}(r,L,M_0,\omega)$ and assume that $\Omega\subset B_R$, for some $R>0$.
Let $\tilde{D}_n=B_{R+1}\setminus K_n$. We assume that $K_n\to K$ in the Hausdorff distance and call $\tilde{D}=B_{R+1}\setminus K$. We prove that
$W^{1,p}(\tilde{D}_n)$ converges to $W^{1,p}(\tilde{D})$ in the sense of Mosco as $n\to\infty$. Condition (M1) of Mosco convergence follows by point \textit{v}) of Proposition~\ref{mosco:m1}. About condition (M2), one can follow the argument of \cite[Theorem 3.9]{rondi:ac}, which is strongly based on the original argument developed in the proof of \cite[Theorem 4.2]{giacomini}. The key point of the proof of Theorem~3.9 in \cite{rondi:ac} was the fact that the boundaries of the chosen hypersurfaces have zero capacity, for $1<p\leq 2$. In our case, we just need to replace the boundary with the singular set. 

Then the proof can be concluded by using Proposition~\ref{mosco:m1} \textit{iv}) and \textit{vi}), since we now know that $W^{1,p}(B_{R+1} \setminus \partial D_n)$ converges in the sense of Mosco to $W^{1,p}(B_{R+1} \setminus \tilde{K})$, with $\tilde{K} \in \mathcal{FR}(r,L,M_0,\omega)$ such that $\partial D_n \to \tilde{K}$ in the Hausdorff distance.
\end{proof}

\begin{remark} For the sake of completeness, let us also mention the following result. Let $\Omega$ be any bounded open set.
Let $r>0$, $L>0$, $M_0\in\N$ and $\omega:(0,+\infty)\to(0,+\infty)$ be a modulus of continuity and let $\{ K_n\}_{n\in\N} \subset \widetilde{\mathcal{FR}}(r,L,M_0,\omega)$ and $K\in \mathcal{K}(\overline{\Omega})$
	 be such that $K_n \to K$ in the Hausdorff distance. We call $D_n=\Omega\setminus K_n$ and $D=\Omega\setminus K$.
	 Let $1 < p \le 2$.  
	Then, provided that $\mathcal{H}^{N-2}(\partial\Omega\cap \partial K)=0$, we have $W^{1,p}(D_n)$ converges to $W^{1,p}(D)$ in the sense of Mosco as $n\to\infty$. In fact, as before, 
condition (M1) of Mosco convergence follows by point \textit{v}) of Proposition~\ref{mosco:m1}, see also Remark~\ref{setminus}, and actually for (M1) there is no need to assume that $\mathcal{H}^{N-2}(\partial\Omega\cap \partial K)=0$. About condition (M2), we start with the sets $\tilde{D}_n=\Omega\backslash \partial K_n$. Then, since $\partial\Omega\cap \partial K$ has zero capacity, we can repeat the argument of Theorem~\ref{teor:finale}. 
Then, as before, we conclude by using Proposition~\ref{mosco:m1} \textit{iv}) and \textit{vi}).

The advantage of this case is that the boundary of $\Omega$, since it is fixed, can be arbitrary, the price to pay is the condition on $\partial\Omega\cap \partial K$ which can not be inferred from the other assumptions.
\end{remark}

\section{A sufficient condition for Sobolev inequalities for nonsmooth domains}\label{Sobolevsec}

We now study Sobolev inequalities. Remarks~\ref{Rellichremark} and \ref{higherMosco} are just two examples of the importance of establishing these inequalities, in particular with constants independent from the domain.

The main tool that we will use to prove the Sobolev inequality is the following version of a Friedrichs type inequality.
For any $u\in L^1_{loc}(\R^N)$ and any $x\in\R^N$, we call $u^+(x)$ and $u^-(x)$, respectively, the approximate $\lim\sup$ and $\lim\inf$ of $u$ at $x$, that is,
$$u^+(x):=\textrm{ap-}\limsup_{y\to x}u(y)=\inf\big\{t\mid \{u>t\}\text{ has density $0$ in }x\big\},$$
and 
$$u^-(x):=\textrm{ap-}\liminf_{y\to x}u(y)=\sup\big\{t\mid \{u>t\}\text{ has density $1$ in }x\big\},$$
see \cite[Definition~1.57]{braides} for details.
For any bounded open set $D$, we consider $\partial D$ as a measure space with the $\mathcal{H}^{N-1}$ measure.

We also use the following \emph{trace inequality}, which can be inferred by results in \cite{adams}. Let $ D\in\mathcal{A}(r,L,R)$. Let $1\leq p<+\infty$.
We consider the following cases
\begin{equation}\label{pqcases}
\left\{
\begin{array}{ll}
\text{a)} &\text{$1\leq p<N$ and $1\leq s\leq \dfrac{(N-1)p}{N-p}$;}\\
\text{b)} &\text{$p=N$ and $1\leq s<+\infty$;}\\
\text{c)} &\text{$p>N$ and $1\leq s\leq+\infty$.}
\end{array}\right.
\end{equation}
Then, there exists a constant $c_{p,s}$, depending on $p$, $s$, $r$, $L$ and $R$ only, such that
\begin{equation}\label{Besovimmersion}
\|u\|_{L^s(\partial D)}\leq c_{p,s}\|u \|_{W^{1,p}( D)}\quad\text{for any }u \in W^{1,p}( D).
\end{equation}

The Friedrichs type inequality is the following.

\begin{proposition} 
\label{friedrichs}
	Let $ D \subset \R^N$ be a bounded open set and let $u\in L^1_{loc}( D)$.
		Let $1\leq p<+\infty$ and assume that $\nabla u \in L^p( D; \R^N)$, that is, $u\in L^{1,p}( D)$. Consider $u$ extended to $0$ outside $ D$. 

		If
	\[ 1\leq p<N, \quad p^{\ast} = \frac{pN}{N-p} , \quad s = \frac{(N-1)p}{N-p}, \]
	then there exists $C_1=C_1(p)>0$, $C_1$ depending only on $p$, such that
	\[ \norm{u}_{L^{p^{\ast}}( D)} \le C_1 \bigl[ \norm{\nabla u}_{L^p( D)} + \norm{ \abs{u}^+ }_{L^s(\partial  D)} \bigr] . \]
	
	If $N\leq p<+\infty$, for any $1<\dfrac {N}{N-1}\leq q<+\infty$, let $1\leq p_1=\dfrac{Nq}{N+q}<N$ and
	$s=\dfrac{(N-1)p_1}{N-p_1}=\dfrac{(N-1)q}{N}$, then
	there exists $C_2=C_2(q)>0$, $C_2$ depending only on $q$, since $C_2(q)=C_1(p_1)$, such that
	\[ \norm{u}_{L^{q}( D)} \le C_2 \bigl[ \norm{\nabla u}_{L^{p_1}( D)} + \norm{ \abs{u}^+ }_{L^s(\partial  D)} \bigr] . \]
\end{proposition}

\begin{proof} The case $1\leq p<N$ is exactly 
	\cite[Corollary 2.4]{rondi:friedrichs}.
	When $p\geq N$, it is enough to note that $q=p_1^{\ast}$.
	\end{proof}

For $1\leq p<N$, this proposition shows that, roughly speaking, the validity of some kind of trace inequality with the optimal value of $s=\dfrac{(N-1)p}{N-p}$
implies the validity of the Sobolev inequality.

\begin{remark}\label{bridge} Let $ D$ be a bounded domain with Lipschitz boundary. Let $1\leq p<+\infty$ and $u\in W^{1,p}(\R^N\backslash \partial D)$.
Then we call $u_+\in L^p(\partial D)$ and $u_-\in L^p(\partial D)$ the traces of $u$ on $\partial  D$ from outside $ D$ and inside $ D$, respectively. Then, for $\mathcal{H}^{N-1}$-a.e. $x\in \partial D$,
we have $u^+=\max\{u_-,u_+\}$ and $u^-=\min\{u_-,u_+\}$, so
$|u|^+(x)=\max\{|u_-|,|u_+|\}\leq |u_-|+|u_+|$.

Moreover, let $u\in L^1_{loc}(\R^N)$, let $x\in \R^N$ and $s>0$ and let $\Phi:B_s(x)\to \R^N$ be a bi-Lipschitz map. If $v=u\circ\Phi^{-1}$, then $u^+(x)=v^+\left(\Phi(x)\right)$.

Let $D\subset B_R$ belong to $\mathcal{O}(r,L,M_0,\omega)$, with $\Omega=B_R$. Let $K=\partial D$. Let $1\leq p<+\infty$ and let $u\in W^{1,p}(D)$. We always extend $u$ to $0$ outside $D$, thus $u\in W^{1,p}(\R^N\setminus K)$ with compact support. If $x\in K\setminus \sing (K)$, through a bi-Lipschitz map $\Phi$, we can map an open neighbourhood $U$ of $x$ to $B_s$, for some $s>0$, such that $\Phi(U\cap K)=B_s\cap\pi$. Let us call $B_s^+=\{y\in B_s\mid y_N>0\}$ 
and $B_s^-=\{y\in B_s\mid y_N<0\}$. Then $v=u\circ\Phi^{-1}\in W^{1,p}(B_s\setminus \pi)$ and we can define its traces on $\pi\cap B_s$ from above and below as $v_+$ and $v_-$. Then, for $\mathcal{H}^{N-1}$-a.e. $y\in U\cap K$ we have $u^+(y)=\max\{v_-(\Phi(y)),v_+(\Phi(y)\}$ and 
$u^-(y)=\min\{v_-(\Phi(y)),v_+(\Phi(y)\}$. Thus, for $\mathcal{H}^{N-1}$-a.e. $x\in K$, $u^+(x)$ and $u^-(x)$ can be defined as the maximum and the minimum, respectively, of this kind of traces from the two sides of $K$.
\end{remark}

Let us introduce the following new classes. We fix a constant $R>0$ and, in the following definitions, we drop the dependence on it.

\begin{definition}
\label{classe:finale-restr}
	A set $K\subset \overline{B_R}$ belongs to $\widehat{\mathcal{FR}}(r,L,M_0,\omega)$ with parameters $r>0$, $L>0$, $M_0\in\N$ and $\omega$ modulus of continuity if 
	\begin{enumerate}[(1)]
	\item $K=\bigcup_{i=1}^M K^i$, where $K^i\in \mathcal{FR}(r,L,M_0)$, with $\Omega=B_R$, for every $i=1,\ldots,M$, for some $M\in\N$.
		\item For any $i=1,\ldots,M$ and $x\in K^i$, if $\dist(x,\bdry(K^i)) = \delta > 0$, then it follows that $\dist(x,\bigcup_{j\neq i} K^j) \ge \omega(\delta)$.
	\end{enumerate}

A set $K\subset \overline{B_R}$ belongs to $\widehat{\mathcal{FR}}_l(r,L,M_0,a,\delta_0)$ with parameters $r>0$, $L>0$, $M_0\in\N$, $a>0$ and $\delta_0>0$ if $K\in \widehat{\mathcal{FR}}(r,L,M_0,\omega)$ with $\omega$ modulus of continuity such that $\omega(\delta)\geq a \delta$ for any $0<\delta\leq\delta_0$.

We finally define
\[ \widehat{\mathcal{O}}_l(r,L,M_0,a,\delta_0) := \{ D \in \mathcal{O}(B_R) \mid  \partial D \in \widehat{\mathcal{FR}}_l(r,L,M_0,a,\delta_0) \}. \]
\end{definition}

We observe that the class $\widehat{\mathcal{FR}}(r,L,M_0,\omega)$ is contained in $\mathcal{FR}(r,L,M_0,\omega)$ and the only difference is that the intersection of the different $\mathcal{FR}(r,L,M_0)$ components can happen only along the boundaries and not along the singular sets. As we shall see, this gives us some more regularity, however the price we pay is that this class may be no longer closed in the Hausdorff distance.
In $\widehat{\mathcal{FR}}_l(r,L,M_0,a,\delta_0)$, we restrict the modulus of continuity $\omega$ to be essentially linear, which allows us to avoid difficult to handle cusps-like structures.  In order to understand such regularity better, we state the next lemma.

\begin{lemma}
\label{incoll:int}
Let us fix parameters $R>0$, $r>0$, $L>0$, $M_0\in\N$, $a>0$ and $\delta_0>0$.
Let $K^1\subset \overline{B_R}$ belong to $\mathcal{FR}(r,L,M_0)$, with $\Omega=B_R$. Let $K^1=\bigcup_{j=1}^J H^j$, where $H^j\in \mathcal{MR}(r,L)$ for every $j=1,\ldots,J$ and $1\leq J\le M_0$. Then there exists a positive constant $a_1$, depending on $r$, $L$ and $R$ only, 
such that for any $x\in H^1$ with $\dist(x, \bdry(H^1)) = \delta>0$ we have $\dist(x, K^1\setminus H^1) \ge a_1\delta$.

Consequently, let $K\subset \overline{B_R}$ belong to $\widehat{\mathcal{FR}}_l(r,L,M_0,a,\delta_0)$, with the representation
$K = \bigcup_{i=1}^M K^i$. Then there exists a positive constant $a_2$, depending on $r$, $L$, $R$ and $a$ only, 
such that for any $x\in H^1$ with $\dist(x, \bdry(H^1)) = \delta$, for some $0<\delta\leq\delta_0$, we have $\dist(x, K\setminus H^1) \ge a_2\delta$.
\end{lemma} 

\begin{proof} Let $x\in H^1$ be such that $\dist(x, \bdry(H^1)) = \delta>0$. We note that $\delta\leq 2R$ and $\dist(x, \bdry(H^1))\leq \dist(x, \bdry(K^1))$.

For the first part, we assume, without loss of generality, that $0<\delta\leq r/4$.
 If $B_r(x) \cap (K^1\setminus H^1) = \emptyset$,
then $\dist(x, K^1\setminus H^1) \ge r\geq \delta$. Otherwise, let $\Phi^1_x$ be the map relative to $K^1$ as a Lipschitz hypersurface.
Then, $\Phi^1_x(x) = 0$ and $B_{\delta/L}\cap \pi$ is contained in the image $\Phi^1_x(H^1 \cap B_r(x))$. Hence,
$K^1\cap B_{\delta/L^2}(x)\subset (\Phi^1_x)^{-1}(B_{\delta/L}\cap \pi)\subset H^1$, thus $\dist(x, K^1\setminus H^1) \ge \delta/L^2$ and the first part is proved.

The second part follows immediately by taking $a_2=\min\{a,a_1\}$.
\end{proof}

We are ready to state the following theorem, the main of this section.

\begin{theorem}
\label{sobolevemb:linear}
	Let D belong to $\widehat{\mathcal{O}}_l(r,L,M_0,a,\delta_0)$
	with parameters $r>0$, $L>0$, $M_0\in\N$, $a>0$ and $\delta_0>0$.
		Let $1\leq p<+\infty$.
		
		If $1\leq p<N$,
	then there exists $C_1>0$, $C_1$ depending only on $p$, $r$, $L$, $M_0$, $a$, $\delta_0$ and $R$, such that
	\[ \norm{u}_{L^{p^{\ast}}(D)} \le C_1 \norm{u}_{W^{1,p}(D)}  . \]
	
	If $N\leq p<+\infty$, then for any $1\leq q<+\infty$ 
	there exists $C_2>0$, $C_2$ depending only on $p$, $q$, $r$, $L$, $M_0$, $a$, $\delta_0$ and $R$, such that
	\[ \norm{u}_{L^q(D)} \le C_1 \norm{u}_{W^{1,p}(D)}  . \]
	\end{theorem}

\begin{proof} It is an immediate consequence of Proposition~\ref{friedrichs} and of the next proposition.\end{proof}

\begin{proposition}
\label{bdryineqprop}
	Let $D$ belong to $\widehat{\mathcal{O}}_l(r,L,M_0,a,\delta_0)$
	with parameters $r>0$, $L>0$, $M_0\in\N$, $a>0$ and $\delta_0>0$.
		Let $1\leq p<+\infty$ and let $u\in W^{1,p}(D)$. 
		Consider $u$ extended to $0$ outside $ D$. 
		
	If
	\[ 1\leq p<N, \quad s = \frac{(N-1)p}{N-p}, \]
	then there exists $\tilde{C}_1>0$, $\tilde{C}_1$ depending only on $p$, $r$, $L$, $M_0$, $a$, $\delta_0$ and $R$, such that
		\[ \norm{ \abs{u}^+ }_{L^s(\partial  D)} \le \tilde{C}_1 \norm{ u}_{W^{1,p}( D)} . \]
	
	If $N\leq p<+\infty$, for any $1\leq s <+\infty$, then
	there exists $\tilde{C}_2>0$, $\tilde{C}_2$ depending 
	only on $p$, $s$, $r$, $L$, $M_0$, $a$, $\delta_0$ and $R$, such that
		\[ \norm{ \abs{u}^+ }_{L^s(\partial  D)} \le \tilde{C}_2 \norm{ u}_{W^{1,p}( D)} . \]
\end{proposition}

\begin{proof} We call $K=\partial D$. Without loss of generality, we assume for simplicity that $u\geq 0$. It is enough to prove the case $1\leq p<N$. In fact, for the other case one can just pick $1\leq p_1=\dfrac{Ns}{N-1+s}<N$ and apply the first case to $p_1$. We just recall that $|D|\leq |B_R|$ and $\mathcal{H}^{N-1}(K)$ is bounded by a constant depending only on $r$, $L$, $M_0$, $a$, $\delta_0$ and $R$, see Remark~\ref{upperbound:finale}.

We take a representation of $K=\bigcup_{i=1}^M K^i$. Since $M\leq \overline{M}$, again by Remark~\ref{upperbound:finale}, it is enough to estimate the norm of $|u|^+$ on $K^i$, for any $i=1,\ldots,M$. Up to reordering, let us just consider $K^1$ and take its representation $K^1=\bigcup_{j=1}^J H^j$. Since $J\leq M_0$, it is enough to estimate the norm of $|u|^+$ on $H^j$, for any $j=1,\ldots,J$. Again up to reordering, we can just focus on proving the estimate on $H^1$, which we recall is belonging to $\mathcal{MR}(r,L)$ with respect to $\Omega=B_R$.

Let $Q^+=[-1,1]^{N-2}\times [0,1]\times [0,1]$,  $Q^-=[-1,1]^{N-2}\times [0,1]\times [-1,0]$  and $\Gamma=[-1,1]^{N-2}\times [0,1]\times  \{0\}\subset \partial Q^{\pm}$. 
Let us fix $x\in H^1$. We begin with the case when $x\in \bdry(H^1)$.
Let $\Phi_x:B_r(x)\to\R^N$ be the bi-Lipschitz map related to $H^1$ as an $\mathcal{MR}(r,L)$ hypersurface and $x$. We have that
$B_{r/L}\cap \pi^+\subset \Phi_x(K\cap B_r(x))$. We can find $\delta_1$, with $0<\delta_1\leq \delta_0$ depending only on $\delta_0$, $r$ and $L$, such that
the diameter of $\delta_1 Q^{\pm}$ is less than or equal to $r/(4L)$. Then, for any $y=(y_1,\ldots,y_{N-1},0)\in \delta_1\Gamma\subset \Phi_x(K\cap B_r(x))$, we have that
$\dist(\Phi_x^{-1}(y),\bdry(H^1))\ge y_{N-1}/L$. Thus, from Lemma~\ref{incoll:int}, $\dist(\Phi_x^{-1}(y),K\setminus H^1)\ge a_2y_{N-1}/L$, therefore
$$\dist(y,\Phi_x((K\setminus H^1)\cap B_r(x)))\geq a_2y_{N-1}/L^2=a_3 y_{N-1}$$
with $a_3=a_2/L^2\leq 1$.

We consider the two wedges $W^{\pm}=\{y\in \delta_1Q^{\pm}\mid  |y_N|<a_3 y_{N-1}\}$ and we conclude that $W^{\pm}\subset
\Phi_x((\R^N\setminus K)\cap B_r(x)))$. Let $v^{\pm}=u\circ \Phi_x^{-1}$ on $W^{\pm}$. Note that $\delta_1\Gamma\subset W^{\pm}$. 
Then, by the trace inequality \eqref{Besovimmersion}, for a constant $c_p$ depending only on $p$, $\delta_1$ and $a_3$, we have
$$\|v_{\pm}\|_{L^s(\delta_1\Gamma)}\leq c_p\|v \|_{W^{1,p}( W^{\pm})}$$
where $v_{\pm}$ denotes as before traces from above or from below respectively.
Hence, since integral and Sobolev norms are preserved under bi-Lipschitz transformations, up to constants depending only on $L$ and, for Sobolev spaces, on $p$, we conclude that
$$\||u|^+\|_{L^s(H^1\cap B_{\delta_1/L}(x))}\leq C\|u \|_{W^{1,p}(\R^N\setminus K)}=C\|u \|_{W^{1,p}(D)}.$$
Here $C$ clearly depends only on $p$, $r$, $L$, $a$, $\delta_0$ and $R$.

Now, if a point $x\in H^1$ has a distance less than or equal to $\delta_1/(2L)$ from $\bdry(H^1)$, we conclude that, for the same constant $C$,
$$\||u|^+\|_{L^s(H^1\cap B_{\delta_1/(2L)}(x))}\leq C\|u \|_{W^{1,p}(D)}.$$

Applying a completely analogous argument, actually easier with respect to the one for boundary points, to the points $x\in H^1$ whose distance 
from $\bdry(H^1)$ is 
greater that $\delta_1/(2L)$, we can find $r_1>0$ and $C_1$, depending only on $p$, $r$, $L$, $a$, $\delta_0$ and $R$, such that for any $x\in H^1$ we have
$$\||u|^+\|_{L^s(H^1\cap B_{r_1}(x))}\leq C_1\|u \|_{W^{1,p}(D)}.$$
By covering $H^1$ with a finite number $m$ of balls of radius $r_1$ centred at points of $H^1$, we conclude that
$$\||u|^+\|_{L^s(H^1)}\leq C_1m^{1/s}\|u \|_{W^{1,p}(D)}.$$
Since $m$ can be bounded by a constant depending on $r_1$ and $R$ only, the proof is concluded.
\end{proof}
 
\section{Stability for the acoustic scattering problem}
\label{section:classscat}

We say that a set $K\subset \R^N$ is a \emph{scatterer} if $K$ is compact and $G:=\R^N\setminus K$ is connected. For any $s>0$, we call $G_s=G\cap B_s=B_s\setminus K$.

Given an open set $D\subset \R^N$, we say that $\sigma$ is a \emph{uniformly elliptic tensor} in $D$ if
$\sigma \in L^{\infty}(D; M_{sym}^{N\times N}(\R))$ and it satisfies the following ellipticity condition for some constants $0<\lambda_0<\lambda_1$
	\begin{equation} \label{ellcond}
	\lambda_0 \abs{\xi}^2 \le \sigma(x) \xi\cdot \xi\le \lambda_1 \abs{\xi}^2 \quad \text{for a.e. } x\in D\text{ and for any }\xi \in \R^N.
	\end{equation}

Given an open set $D$, a uniformly elliptic tensor $\sigma$ in $D$, $\tilde{q}\in L^{\infty}(D)$, $f\in L^2(D)$ and $F\in L^2(D;\R^N)$,
we consider the following equation with variable coefficients
\begin{equation}
\label{helmholtz:var}
	\divergence (\sigma \nabla u) + \tilde{q} u = \divergence(F)+f \qquad \text{in }D
	\end{equation}
whose weak formulation is that $u\in W^{1,2}_{loc}(D)$ and
\[ \int_{D}\sigma\nabla u\cdot \nabla v-\int_{D}\tilde{q}uv=\int_{D}F\cdot\nabla v-\int_{D}fv \] 
for any $v\in W^{1,2}(D)$ with compact support in $D$. If we consider a homogeneous Neumann condition, namely,
\begin{equation}
\label{helmholtz:varBC}
	\sigma \nabla u\cdot \nu=0  \qquad \text{on }\partial D,
	\end{equation}
the weak formulation of \eqref{helmholtz:var}-\eqref{helmholtz:varBC} is
	that $u\in W^{1,2}(D)$ and
\begin{equation}\label{weakform}
\int_{D}\sigma\nabla u\cdot \nabla v-\int_{D}\tilde{q}uv=\int_{D}F\cdot\nabla v-\int_{D}fv\quad\text{for any }v\in W^{1,2}(D).
\end{equation}

We have the following stability result for boundary value problems of this kind under variations of $D$, the coefficients $\sigma$ and $q$, and the 
source terms $f$ and $F$.

\begin{proposition}\label{stabatfinite}
Let $\Omega$ be an arbitrary bounded domain in $\R^N$ and let us fix $0<\lambda_0<\lambda_1$ and $c_1>0$. Suppose that, for any $n\in\N$, $D_n\in \mathcal{O}(\Omega)$, $\sigma_n$ is a uniformly elliptic tensor in $\Omega$ satisfying \eqref{ellcond} in $\Omega$, $\tilde{q}_n\in L^{\infty}(\Omega)$ with
$\|\tilde{q}_n\|_{L^{\infty}(\Omega)}\leq c_1$, $f_n\in L^2(\Omega)$ and $F_n\in L^2(\Omega;\R^N)$. Let $u_n\in W^{1,2}(D_n)$ solve the corresponding problem \eqref{helmholtz:var}-\eqref{helmholtz:varBC}, that is, for any $v\in W^{1,2}(D)$,
\begin{equation}\label{weakformbis}
\int_{D_n}\sigma_n\nabla u_n\cdot \nabla v-\int_{D_n}\tilde{q}_nu_nv=\int_{D_n}F_n\cdot\nabla v-\int_{D_n}f_n v.
\end{equation}

We assume that $D_n\to D$ in the Hausdorff complementary topology, $W^{1,2}(D_n)$ converges to $W^{1,2}(D)$ in the sense of Mosco, $\sigma_n\to \sigma$ and $\tilde{q}_n\to \tilde q$ almost everywhere in $\Omega$,
$f_n$ converges to $f$ weakly in $L^2(\Omega)$ and $F_n$ converges to $F$ weakly in $L^2(\Omega;\R^N)$.

Moreover, we assume that, for some constant $C>0$, $\|u_n\|_{L^2(D_n)}\leq C$ for any $n\in\N$. 

Then, up to a subsequence, $(u_n,\nabla u_n)$ converges to $(u,\nabla u)$, with the usual extensions to $0$, weakly in $L^2(\Omega;\R^{N+1})$, where $u\in W^{1,2}(D)$ solves \eqref{helmholtz:var}-\eqref{helmholtz:varBC}. Furthermore,
we also have that $\sigma_n\nabla u_n$ converges to $\sigma\nabla u$  weakly in $L^2(\Omega;\R^{N})$ and $\sqrt{\sigma_n}\nabla u_n$ converges to $\sqrt{\sigma}\nabla u$
weakly in $L^2(\Omega;\R^{N})$. 
\end{proposition}

\begin{proof} It is easy to infer that we have a uniform bound on $\|u_n\|_{W^{1,2}(\Omega)}$, therefore, up to a subsequence, $(u_n,\nabla u_n)$ weakly converges, in $L^2(\Omega;\R^{N+1})$, to $(u,\nabla u)$, with $u\in W^{1,2}(D)$. Here we used condition (M1) of Mosco convergence. It is not difficult to show that also $\sigma_n\nabla u_n$ converges to $\sigma\nabla u$  weakly in $L^2(\Omega;\R^{N})$ and $\sqrt{\sigma_n}\nabla u_n$ converges to $\sqrt{\sigma}\nabla u$
weakly in $L^2(\Omega;\R^{N})$. 

We need to prove that $u$ solves \eqref{helmholtz:var}-\eqref{helmholtz:varBC}. Take $v\in W^{1,2}(D)$ and, by condition (M2) of Mosco convergence, let
$v_n\in W^{1,2}(D_n)$ be such that $(v_n,\nabla v_n)$ converges to $(v,\nabla v)$ strongly in $L^2(\Omega;\R^{N+1})$. Then
\[ \int_{D_n}\sigma_n\nabla u_n\cdot \nabla v_n-\int_{D_n}\tilde{q}_nu_nv_n=\int_{D_n}F_n\cdot\nabla v_n-\int_{D_n}f_n v_n\]
and it is not difficult to check that we can pass to the limit and obtain \eqref{weakform}.
\end{proof}

\begin{corollary}\label{stabatfinitecor}
Under the assumptions of Proposition~\ref{stabatfinite}, further assume that there exists $p>2$ and $C_1>0$ such that for any $n\in \N$
\begin{equation}\label{uniformhigher}
\|v\|_{L^p(D_n)}\leq C_1\|v\|_{W^{1,2}(D_n)}\quad\text{for any }v\in W^{1,2}(D_n).
\end{equation}
Then we have that, up to a subsequence, $u_n\to u$ strongly in $L^2(\Omega)$.

Furthermore, if $F_n\to F$ strongly in $L^2(\Omega;\R^N)$, we also have that, for the same subsequence,
 $\sqrt{\sigma_n}\nabla u_n\to \sqrt{\sigma}\nabla u$ strongly in $L^2(\Omega;\R^{N})$.
In this case, we can also replace the assumption that $\tilde{q}_n\to \tilde{q}$ almost everywhere in $\Omega$ with the assumption that
$\tilde{q}_n\stackrel{\ast}{\rightharpoonup} \tilde{q}$ with respect to weak-$\ast$ convergence in $L^{\infty}(\Omega)$.
\end{corollary}	

\begin{proof} The first part immediately follows by Remark~\ref{higherMosco}. Then, provided $F_n\to F$ strongly in $L^2(\Omega;\R^N)$,
by the weak formulation \eqref{weakformbis} with $v=u_n$, we infer that
$$\int_{D_n}\sigma_n\nabla u_n\cdot\nabla u_n\to\int_{D}\sigma\nabla u\cdot\nabla u\quad\text{as }n\to\infty$$
which guarantees the strong convergence of $\sqrt{\sigma_n}\nabla u_n$ to $\sqrt{\sigma}\nabla u$. About the last remark, if $u_n\to u$ and $v_n\to v$ strongly in $L^2(\Omega)$, then it is enough that $\tilde{q}_n\stackrel{\ast}{\rightharpoonup} \tilde{q}$ with respect to weak-$\ast$ convergence in $L^{\infty}(\Omega)$ to obtain that
$$\int_{D_n}\tilde{q}_nu_nv_n\to \int_D\tilde{q}uv\quad\text{as }n\to\infty$$
and the proof is concluded.
\end{proof}

\begin{remark}
If we further assume that \eqref{helmholtz:var}-\eqref{helmholtz:varBC} admits a unique solution, then the convergences in Proposition~\ref{stabatfinite} and
in Corollary~\ref{stabatfinitecor} hold
without passing to subsequences.
\end{remark}

\begin{definition}
	The equation $\divergence (\sigma\nabla u)+qu=0$, with $\sigma$ uniformly elliptic tensor in $D$ and $q\in L^{\infty}(D)$, satisfies the \emph{unique continuation property} (UCP) in an open connected set $D$ if for any solution $u$ that vanishes in an open not empty subset $D_1 \subset D$, it follows that $u \equiv 0$ in all of $D$.
		\end{definition}

\begin{remark}
For $N=2$, $\sigma$ uniformly elliptic tensor and $q\in L^{\infty}$ are enough to guarantee the UCP; in particular, no regularity of $\sigma$ is required. For $N\geq 3$, some regularity of $\sigma$ is needed, indeed a sufficient assumption is that $\sigma$ is locally Lipschitz in $D$, see for instance \cite{three:spheres}.
\end{remark}

Following \cite{ball}, and the generalisation in \cite{rondi:em}, we show that UCP holds if $\sigma$ is piecewise Lipschitz in the following sense.

\begin{proposition}
\label{ucp:prop}
	Let $D\subset \R^N$ be an open set and $\sigma$ be a uniformly elliptic tensor in $D$ and $q\in L^{\infty}(D)$.
	Assume that there exists a family $\{D_i\}_{i\in I}$ of pairwise disjoint domains contained in $D$ such that
		\[ D \subset \bigcup_{i\in I} \overline{D_i} \]
		and $\abs{\Sigma} = 0$ where 
		\[ \Sigma = D \cap \bigcup_{i\in I} \partial D_i. \]
		We say that $x \in \Sigma$ exactly separates two parts if there exist $\delta > 0$ and two different indices $i$ and $j$ such that
		$\abs{ B_{\delta}(x) \setminus (D_i \cup D_j)} = 0 $
and $B_{\delta}(x) \cap D_i$ and $B_{\delta}(x) \cap D_j$ are not empty. Then we assume that $\tilde{D} = D \setminus C$ is connected, where
		\[ C = \{ x \in \Sigma \mid \text{$x$ does not exactly separate any two parts} \}. \]
		Finally, we assume that for any compact set $K$ contained in $D$ we have that $\sigma|_{K\cap D_i}$ is Lipschitz for any $i\in I$.
		
	Then, the equation
	$\divergence (\sigma \nabla u) + k^2 q u = 0 $
	satisfies the UCP in $D$.
\end{proposition}

\begin{proof} The argument is exactly the same used to prove \cite[Proposition 2.13]{rondi:em}. We just note that under our assumptions, $D$ is necessarily connected and the set of indexes $I$ is at most numerable.
\end{proof}

\begin{remark}
	By \cite[Lemma 2.14]{rondi:em}, we can replace the assumption that $\tilde{D} = D \setminus C$ is connected with the simpler one that 
$D$ is connected and $\mathcal{H}^s(C) < +\infty$ for some $s < N-1$.
\end{remark}

The main purpose of this section is to prove the stability of the scattering problem 
\begin{equation}
	\label{scattering:var}
	\left\{ \begin{array}{ll}
		\divergence (\sigma \nabla u) + k^2 q u = 0 &\text{in } G=\R^N \setminus K \\
		u = u^i + u^s  &\text{in } G \\
		\sigma \nabla u \cdot \nu = 0 & \text{on }\partial G= \partial K \\
		\displaystyle\lim_{r\to +\infty} r^{(N-1)/2} \left( \frac{\partial u^s}{\partial r} - \mathrm{i}k u^s \right)=0 & \text{with } r = \abs{x}  
		\end{array}
	\right.
\end{equation}
with respect to the \emph{sound-hard} scatterer $K$, the coefficients $\sigma$ and $q$, the \emph{incident field} $u^i$ and the wavenumber $k>0$.
The function $u$ is called the \emph{total field}, whereas $u^s$ is called the \emph{scattered field}.
About the coefficients, we assume $\sigma$ to be a uniformly elliptic tensor in $G$, whereas $q\in L^{\infty}(G)$ satisfies, for some constants $0<c_0<c_1$, 
	\begin{equation} \label{ellcond2}
	c_0  \le q(x) \le c_1 \quad \text{for a.e. } x\in G.
	\end{equation}
We always assume that the space is isotropic outside a ball $\overline{B_{R_0}}$, for some $R_0 >0$, that is,
	\begin{equation}\label{isotropic}
	\sigma \equiv I_N \text{ and } q \equiv 1 \quad \text{in } \R^N \setminus \overline{B_{R_0}}.
	\end{equation}
The limit in the last line of \eqref{scattering:var} is the so-called \emph{Sommerfeld radiation condition} and it means that the scattered field is, outside $\overline{B_{R_0}}$, a radiating solution to the Helmholtz equation $\Delta u+k^2u=0$.
The incident field is typically an entire solution to the Helmholtz equation $\Delta u+k^2u=0$, here we limit ourselves to consider $u^i$ as a \emph{plane wave} with wavenumber $k$ and \emph{direction of propagation} $d\in \mathbb{S}^{N-1}$, namely
\begin{equation}\label{planewave}
u^i(x):=\rme^{\mathrm{i}kx\cdot d}\quad \text{for any }x\in \R^N.
\end{equation}

\begin{remark} About the scattering problem \eqref{scattering:var}, we refer to \cite{colton:kress,wilcox}. We note that,
under the previous assumptions, if $K\subset \overline{B_R}$, for some $R>0$, we say that $u$ is a weak solution to the exterior boundary value problem \eqref{scattering:var} if, for any $s>R$, $u\in W^{1,2}(G_s;\C)$, it satisfies
\begin{multline*}
\int_{G_s}\sigma\nabla u\cdot \nabla v-k^2\int_{G_s}quv=0\\\text{for any }v\in W^{1,2}(G_s;\C)\text{ such that }v=0\text{ on }\partial B_s,
\end{multline*}
and $u^s=u-u^i$ satisfies the limit in the Sommerfeld radiation condition, which has to hold uniformly with respect to all directions. 

We note that a weak solution to \eqref{scattering:var} is unique provided the Helmholtz equation 
$\divergence (\sigma \nabla u) + k^2 q u = 0$ has the UCP in $G=\R^N\setminus K$. Concerning the existence of a solution, we still need the UCP property of the equation together with a mild regularity of $K$, namely that $G_{R+1}=B_{R+1}\setminus K$ satisfies the RCP$_2$.
\end{remark}

Unfortunately, for stability, following \cite{rondi:ac}, RCP$_2$ is not enough and a stronger assumption is required, namely a uniform higher integrability of $W^{1,2}(G_{R+1})$ functions such as in Corollary~\ref{stabatfinitecor}. The following stability result just extends to the variable coefficients case Proposition~2.15 in \cite{rondi:ac}. The proof is completely analogous, once Proposition~\ref{stabatfinite} and Corollary~\ref{stabatfinitecor} are established, and we omit it.

\begin{theorem}[Stability]\label{teostab}
For any $n\in\N$, let $K_n$ be a scatterer contained in $\overline{B_R}$, for some fixed $R$. We call $D_n=B_{R+1}\setminus K_n$ and we assume that
there exists $p>2$ and $C_1>0$ such that for any $n\in \N$
\eqref{uniformhigher} holds.

Let $\sigma_n$ be a uniformly elliptic tensor in $\R^N$ satisfying \eqref{ellcond} with constants $0<\lambda_0<\lambda_1$, and let $q_n\in L^{\infty}(\R^N)$
satisfying \eqref{ellcond2} with constants $0<c_0<c_1$. We also assume that, for some $R_0>0$, $\sigma_n$ and $q_n$ satisfy \eqref{isotropic}.

Let $d_n\in \mathrm{S}^{N-1}$ and $k_n>0$, and let the incident field be given by $u^i_n(x)=
\rme^{\mathrm{i}k_nx\cdot d_n}$, $x\in \R^N$.

Let $u_n$ be a weak solution to
$$	\left\{ \begin{array}{ll}
		\divergence (\sigma_n \nabla u_n) + k_n^2 q_n u_n = 0 &\text{in } G^n=\R^N \setminus K_n \\
		u_n = u_n^i + u_n^s  &\text{in } G^n \\
		\sigma_n \nabla u_n \cdot \nu = 0 & \text{on }\partial G^n= \partial K_n \\
		\displaystyle\lim_{r\to +\infty} r^{(N-1)/2} \left( \frac{\partial u_n^s}{\partial r} - \mathrm{i}k_n u_n^s \right)=0 & \text{with } r = \abs{x}. 
		\end{array}
	\right.
$$

 We assume that $K_n$ converges to $K$ in the Hausdorff distance, $W^{1,2}(D_n)$ Mosco converges to $W^{1,2}(D)$, with $D=B_{R+1}\setminus K$,   $\sigma_n$ converges to $\sigma$ almost everywhere in $B_{R_0}$, $q_n$ converges to $q$ almost everywhere in $B_{R_0}$ or 
 with respect to weak-$\ast$ convergence in $L^{\infty}(B_{R_0})$, $d_n$ converges to $d\in\mathbb{S}^{N-1}$ and $k_n$ converges to $k\in \R$,
 
 We further assume that $K$ is a scatterer, $\divergence (\sigma \nabla u) + k^2 q u = 0$ satisfies the UCP in $G=\R^N\setminus K$, and that, only if $N=2$, $k>0$. 
 
Then $(u_n,\sqrt{\sigma_n}\nabla u_n)$ converges to $(u,\sqrt{\sigma}\nabla u)$ strongly in $L^2(B_s;\R^{N+1})$ for any $s>0$, where
$u$ solves \eqref{scattering:var} with incident field $u^i$ given by \eqref{planewave}.
\end{theorem}

\begin{remark} For $N\geq 3$, we allow the limit $k$ to be equal to $0$. For $k=0$, the Sommerfeld radiation condition has to be replaced by
$$u^s=o(1)\quad\text{as }r=|x|\to+\infty,$$
with the limit holding uniformly with respect to all directions.
For $N=2$, instead, $k$ must be strictly positive, since the limit for small wavenumbers is troublesome. For further details, see \cite{kress} and  the discussion in \cite[Section~3]{rondi:uniquedet}.
\end{remark}

In the sequel we construct classes of scatterers and of coefficients satisfying the assumptions of Theorem~\ref{teostab} and such that, up to subsequences, they satisfy the convergences required in Theorem~\ref{teostab}.

\begin{definition}
A set $K\subset \overline{B_R}$ belongs to $\mathcal{SC}_l(R,r,L,M_0,a,\delta_0,\gamma)$, with parameters $r>0$, $L>0$, $M_0\in\N$, $a>0$, $\delta_0>0$ and $\gamma$ modulus of continuity, if $K$ is compact, $\partial K\in \widehat{\mathcal{FR}}_l(r,L,M_0,a,\delta_0)$ and
	$K$ satisfies the \emph{uniform exterior connectedness property} with function $\gamma$, that is, for any $t>0$, for any $x_1,x_2 \in \R^N$ such that $B_t(x_1)$ and $B_t(x_2)$ are contained in $\R^N \setminus K$ and for any $0<s<\gamma(t)$, one can find a smooth curve $\Gamma$, for instance of regularity $\mathcal{C}^1$, which connects the points $x_1$ and $x_2$ and such that $B_s(\Gamma)$ is still contained in $\R^N \setminus K$.
\end{definition}

Observe that the uniform exterior connectedness property implies that $\R^N \setminus K$ is connected. Moreover, this condition is also closed with respect to convergence in the Hausdorff distance, see
\cite[Lemma 2.5]{rondi:uniquedet}.

\begin{definition}
\label{class:N}
	We say that $$(\sigma,q,k)\in\mathcal{N}=\mathcal{N}(R_0, r,L,M_0,\omega,\lambda_0,\lambda_1,c_0,c_1,\underline{k},\overline{k})$$ with parameters $R_0>0$, $r>0$, $L>0$, $M_0\in \N$, $0<\lambda_0<\lambda_1$, $0<c_0<c_1$, $0<\underline{k}<\overline{k}$ and $\omega$ modulus of continuity if the following holds.
	First, $\sigma\in L^{\infty}(\R^N;M_{sym}^{N\times N}(\R))$ is a uniformly elliptic tensor in $\R^N$ with constants $0<\lambda_0<\lambda_1$, that is, \eqref{ellcond} holds with $D=\R^N$, $q\in L^{\infty}(\R^N)$ satisfies \eqref{ellcond2} with $D=\R^N$, and \eqref{isotropic} holds. The number $k\in\R$ satisfies
	 \begin{equation}\label{kcondit}
 0< k < \overline{k}\text{ if }N \ge 3\text{ or }0 < \underline{k} < k < \overline{k}\text{ if }N = 2.
 \end{equation}
	Finally, $\sigma$ has the following regularity.
If $N=2$, we assume that any entry of the matrix $\sigma$ has total variation on $B_{R_0}$ bounded by $L$.
	If $N\geq 3$, we assume that there exists $K \in \mathcal{FR}(r,L,M_0,\omega)$, with respect to $\Omega=B_{R_0}$, depending on $\sigma$, satisfying the following properties. Let $D_0$ be the unbounded connected component of $\R^N \setminus K$ and let $D_i$, $i=1,\ldots,\tilde{M}$, be the bounded ones. Then, $\sigma|_{D_0} \equiv I_N$, while $\sigma|_{D_i}$ is Lipschitz with Lipschitz constant bounded by $L$, for every $i=1,\ldots,\tilde{M}$.
\end{definition}

We note that $\tilde{M}$ depends on $K$, thus on $\sigma$, but, by Lemma~\ref{raggio:finale}, it is bounded by a constant depending on $R$, $r$, $L$, $M_0$ and $\omega$ only. Moreover, for any $(\sigma,q,k)$ in such a class, the equation
$\divergence(\sigma \nabla u)+k^2qu=0$ satisfies the UCP, thanks to Proposition~\ref{ucp:prop}.

\begin{proposition}\label{classNconverg}
For any $n\in\N$, let $(\sigma_n,q_n,k_n)\in \mathcal{N}$, where
$\mathcal{N}$ is as in Definition~\textnormal{\ref{class:N}}, with given parameters.
Then, up to a subsequence, $\sigma_n\to \sigma$ almost everywhere in $\R^N$ and also in $L^p(\mathbb{R}^N)$ for any $1\leq p<+\infty$, $q_n$
converges to $q$ with respect to weak-$\ast$ convergence in $L^{\infty}(B_{R_0})$ and $k_n\to k\in \R$ where $\sigma$, $q$ and $k$ satisfy the following properties.

We have that $\sigma$ is a uniformly elliptic tensor satisfying \eqref{ellcond} with $D=\R^N$, 
$q\in L^{\infty}(\R^N)$ satisfies \eqref{ellcond2} with $D=\R^N$, and \eqref{isotropic} holds. The number $k\in\R$ is such that $0\leq k\leq \overline{k}$, with
$k\geq \underline{k}>0$ if $N=2$. Finally, the equation
$\divergence(\sigma \nabla u)+k^2qu=0$ satisfies the UCP.
\end{proposition} 

\begin{proof} About $q$ and $k$, the listed convergence properties are straightforward.
About $\sigma$, for $N=2$,  boundedness in $BV$ implies compactness in $L^1$, so also this case is simple, since from convergence in $L^1$ we can deduce, due to uniform boundedness, convergence in $L^p$ for any $p$ finite and, up to a further subsequence, convergence almost everywhere.

About $\sigma$, for $N\geq 3$, we argue like in \cite[Lemma~5.3]{rondi:em}. Due to uniform boundedness, it is enough to prove $L^1$ convergence or almost everywhere convergence.
We can assume, passing to subsequences, that the sets $K_n$ associated to $\sigma_n$ converge in the Hausdorff distance to
$K \in \mathcal{FR}(r,L,M_0,\omega)$. We can also assume that $\tilde{M}_n=\tilde{M}$ for any $n\in \N$. We call $D_i^n$, $i=0,\ldots,\tilde{M}$, the connected components of $\R^N\setminus K_n$, $D_0^n$ being the unbounded one. We can assume that, for any $i=0,\ldots,\tilde{M}$,
$K_n\cup D_i^n$ converges to $G_i$ in the Hausdorff distance. By Proposition~\ref{convh:finale2}, we have that $K_n$ converges to $\partial G_i$ in the Hausdorff distance, and we call $D_i$ the interior of $G_i$. We note that the sets $D_i$ are open, pairwise disjoints and their union is $\R^N\setminus K$.
However, $D_i$ may be the union of more than one connected component of $\R^N\setminus K$.

By
\cite[Theorem 3.1.1]{evans}, we call $\sigma^n_i$, for any $i=1,\ldots,\tilde{M}$ and for any $n\in \N$, an extension of $\sigma_n|_{D^n_i}$ such that
$\sigma^n_i$ is Lipschitz on $\overline{B_{R_0}}$, with Lipschitz constants bounded by a constant, depending only on $N$, times $L$. Still up to subsequences, we can assume that $\sigma^n_i\to \sigma_i$ uniformly on $\overline{B_{R_0}}$, for any $i=1,\ldots,\tilde{M}$. Clearly $\sigma_i$
is Lipschitz on $\overline{B_{R_0}}$.

We take as $\sigma$ the tensor such that $\sigma|_{D_i}=\sigma_i$.
Now let $x\in\R^N\backslash K$ and let $x\in D_i$ for some $i$. We have that, for some $\delta>0$ and any $n\geq n_0$, $B_{\delta}(x)\subset D_i^n$, thus
$\sigma^n(x)=\sigma^n_i(x)\to \sigma(x)$. Since $|K|=0$, we have obtained convergence almost everywhere. By convergence almost everywhere, it is clear that
$\sigma$ is a uniformly elliptic tensor satysfing \eqref{ellcond}. Again by Proposition~\ref{ucp:prop},
we conclude that
$\divergence(\sigma \nabla u)+k^2qu=0$ satisfies the UCP.
\end{proof}

We conclude with the following uniform bound result which follows easily from our previous analysis and the arguments developed in \cite{rondi:ac,rondi:em}.

\begin{theorem}[Uniform bounds]
\label{unif:bounds}
	Let $R>0$, $r>0$, $L>0$, $M_0 \in \N$, $a>0$, $\delta_0>0$, $\gamma$ modulus of continuity, $R_0>0$, $\omega$ modulus of continuity, 
	$0<\lambda_0<\lambda_1$, $0<c_0<c_1$ and $0<\underline{k}<\overline{k}$.
	For any $K\in\mathcal{SC}_l(R,r,L,M_0,a,\delta_0,\gamma)$, any $(\sigma,q,k)\in\mathcal{N}(R_0, r,L,M_0,\omega,\lambda_0,\lambda_1,c_0,c_1,\underline{k},\overline{k})$ and any $d\in\mathbb{S}^{N-1}$, let $u$ be the solution
	to \eqref{scattering:var}, with $u^i$ given by \eqref{planewave}.

	Then, for any $s>0$, there exists $C_1>0$, depending only on the parameters above and on $s$, such that
	\begin{equation}
		\norm{u}_{L^2(B_s \setminus K)}+\norm{\nabla u}_{L^2(B_s \setminus K)} \le C_1.
		\end{equation}
		Furthermore, there exists a constant $C_2$, depending only on the parameters above, such that
	\begin{equation}
	\abs{u^s(x)}\leq \frac{C_2}{|x|^{(N-1)/2}}\quad\text{for any }x\in \R^N\text{ with }|x|\geq \max\{R_0,R\}+1.
		\end{equation}
		
\end{theorem}

\begin{remark} For $N\geq 3$, we can just set $\underline{k}=0$ both in Definition~\ref{class:N} and in Theorem~\ref{unif:bounds}.
\end{remark}

\appendix

\section{Comparison with the class $\mathcal{G}(r,L,C)$}

Here we show that classes $\mathcal{MR}(r,L,M_0)$ generalise classes $\mathcal{G}(r,L,C)$ introduced in \cite{giacomini}.

\subsection{Cone condition}

First of all, we fix some notation. Let $N\ge 1$, let $B_1,B_2\subset\R^N$ be two open balls such that $B_1$ is centred in $0$ and $0\notin B_2$ and let $x\in\R^N$. We denote \emph{finite closed cone} in $\R^N$ with vertex in $x$ a set which can be described as
\[ 	C_x := x + (\overline{B_1} \cap \{\lambda  y \mid y\in \overline{B_2}, \lambda \ge 0\}). \]
Instead, let $y_1,\ldots,y_N\in\R^N$ be $N$ linearly independent vectors and let $x\in\R^N$. We denote \emph{parallelepiped} in $\R^N$ with vertex in $x$ a set such as
	\[	P_x := x + \left\{ \sum_{j=1}^N \lambda_j y_j \,\Big|\, 0 \le \lambda_j \le 1 \text{ for every } 1\le j\le N \right\}	.\]
Moreover, the \emph{centre} of $P_x$ is $c(P_x) = x + \frac{1}{2} (y_1 + \ldots + y_N)$. Finally, we say that two cones $C$ and $C'$ are \emph{congruent}, and we write $C \cong C'$, if there exists a rototranslation $\Psi$ of $\R^N$ such that $\Psi(C) = C'$ and, analogously, $P \cong P'$ if  $\Psi(P) = P'$. It is clear that every finite closed cone with vertex in $x$ contains a parallelepiped with vertex in $x$ and viceversa.

Since we mostly work with compact sets, we need a modified version of the classical cone property for open sets, as in \cite[Definition 4.3]{adams}.

\begin{definition}
	\label{cone:cond}
	Let $C$ be a finite closed cone in $\R^N$ with vertex in the origin. We say that a compact set $K\subset\R^N$ satisfies the \emph{cone condition with respect to $C$} if for any $x\in K$ there exists a finite closed cone $C^x$ congruent to $C$ such as $x\in C^x \subset K$.
\end{definition} 

Other than using closed cones and not open ones, the main difference with the classical cone property is that we do not require the cone $C^x$ to have vertex in the assigned point $x$. This gives more flexibility thanks to the following lemma.

\begin{lemma}
\label{unione:coni}
	Let $C$ be a finite closed cone in $\R^N$ with vertex in the origin and let $K=\bigcup_{i\in I} C_i$ where $I$ is a set of indices and $C_i \cong C$ for every $i \in I$. Suppose additionally that the vertices of the cones $C_i$ can vary only in a compact set $\overline{B}$. \par
	Then, there exists $\{ C_j \}_{j\in J}$, where $J$ is a set of indices, $C_j \cong C$ for any $j\in J$ and $C_j$ has vertex in $\overline{B}$ for any $j\in J$, such that $\overline{K} = \bigcup_{j\in J} C_j $. Namely, the closure of a union of closed cones is still a union of closed cones. Furthermore, $\overline{K}$ satisfies the cone condition with cone $C$.
\end{lemma}

\begin{proof}
Observe that
the set of closed cones in $\R^N$ congruent to $C$ with vertex in $\overline{B}$, equipped with the metric given by Hausdorff distance,
is compact since it is homeomorphic to the product space $\overline{B} \times \mathbb{S}^{N-1}$.
		
		Let $x\in \overline{K}$, then there exists $\{x_n\}_{n\in\N} \subset K$ such that $x_n\to x$. Since $x_n\in K$, for any $n\in\N$ there exists $ C_{i_n}$ such that $x_n \in C_{i_n}$. Now, by compactness, up to subsequences, $C_{i_n} \to C^x$ in the Hausdorff distance. Then, $C_{i_n} \to C^x$ implies that $C^x \subset \overline{K}$ and that $x\in C^x$. Hence, it is enough to define $J=\overline{K}$, $j=x\in \overline{K}$ and $C_j = C^x$ to obtain that $\overline{K} = \bigcup_{x\in \overline{K}} C^x$. Note that any of these $C^x$ has vertex in $\overline{B}$, so it is immediate to conclude that $\overline{K}$ is compact and that 
		it satisfies the cone condition with cone $C$.
\end{proof}

The following modification of the Gagliardo theorem allows us to decompose compact sets with the cone condition into a finite union of compact sets with Lipschitz boundary, where the number and the regularity of these Lipschitz sets depend only on the cone $C$.  

\begin{theorem}[Gagliardo]
\label{gagliardo}
	Let $C \subset\R^N$ be a finite closed cone with vertex in $0$ and let $K\subset\R^N$ be a compact set with $\diam(K)\le d$ satisfying the cone condition with respect to $C$. Then, for any $ \rho > 0$ there exist $ A_1,\ldots,A_l$ compact subsets of $K$ with $\diam(A_i)\le\rho$ for every $i=1,\ldots,l$ and there exist $ P_1,\ldots,P_l$ congruent parallelepipeds with vertex in $0$ such that the following holds:
	\begin{enumerate}[a\textnormal{)}]
		\item For any $x\in K$ there exists $1\le i \le l$ such that $x+P_i \subset C^x$.
		\item $K = \bigcup_{i=1}^l K_i$ with $K_i = \bigcup_{x\in A_i} (x + P_i)$ compact for every $i=1,\ldots,l$.
	\end{enumerate}
	In particular, the number $l\in\N$ and the parallelepipeds $P_1,\ldots,P_l$ depend only on $C$, $d$ and $\rho$ and not on $K$ specifically.
	
	Moreover, there exist $\bar{\rho} >0$, $\eta >0$ and $L_1>0$, depending only on $C$, such that for any $0<\rho\leq \bar{\rho}$ we have
		that for every $i=1,\ldots,M$ the set $K_i$ is the closure of an open set whose boundary is Lipschitz with constants $\eta$ and $L_1$.
		\end{theorem}

\begin{proof}
	It is easy to modify the argument used in \cite[Theorem 4.8]{adams} to our case of compact sets with the cone condition.
\end{proof}

\subsection{Class $\mathcal{G}(r,L,C)$}
We now introduce the class of hypersurfaces originally defined in \cite{giacomini}.

\begin{definition}
\label{classe:grlc}
	We say that a Lipschitz hypersurface $K\in\mathcal{K}(\overline{\Omega})$ belongs to the class $\mathcal{G}(r,L,C)$ with parameters $r>0$, $L>0$ and $C\subset \R^{N-1}$ finite closed cone with vertex in $0$, if, in addition to hypotheses (a) and (b) of Definition \ref{ipersup:lip}, it satisfies for any $x\in K$ the property
	\begin{enumerate}[(a$'$)]\setcounter{enumi}{2}
		\item for any $y\in B_{\frac{r}{2}}(x) \cap K$ there exists a finite closed cone $C_y\subset \pi$ with $C_y \cong C$ such that $ \Phi_x (y) \in C_y \subset \Phi_x(\overline{B_r(x)}\cap K) $.
	\end{enumerate}
\end{definition}

\begin{remark}
	In the original definition in \cite{giacomini}, property (c$'$) was slightly different, namely
	\begin{enumerate}[(a$''$)]\setcounter{enumi}{2}
		\item for any $y\in B_{\frac{r}{2}}(x) \cap K$ there exists a finite closed cone $C_y\subset \pi$ with $C_y \cong C$ such that $ \Phi_x (y) \in C_y \subset \Phi_x(B_r(x)\cap K) $.
	\end{enumerate}
	However, these two conditions are equivalent, up to changing the cone $C$. Indeed, (c$''$) obviously implies (c$'$). Viceversa, if we choose $\tilde{C}\subset C$ such that $\diam(\tilde{C})< r/(2L)$, we have that $y\in \tilde{C}_y \subset C_y \subset \Phi_x(\overline{B_r(x)}\cap K)$. However, $\tilde{C}_y \subset B_{r/(2L)}(y)$, so that $\Phi_x^{-1}(\tilde{C}_y) \subset B_{r/2}(\Phi_x^{-1}(y)) \subset B_r(x)$, which in turn implies that $\tilde{C}_y \subset \Phi_x(B_r(x)\cap K)$. Our modification on property (c$'$) is justified by the fact that it allows to prove that the class $\mathcal{G}(r,L,C)$ is closed, thus compact, with respect to the Hausdorff distance.
\end{remark}

\begin{proposition}
\label{convh:giac}
	Let $r>0$, $L>0$ and $C\subset \R^{N-1}$ finite closed cone with vertex in $0$ and let $\{ K_n\}_{n\in\N} \subset \mathcal{G}(r,L,C)$ such that $K_n \to K$ in the Hausdorff distance, with $K\in\mathcal{K}(\overline{\Omega})$.  
	Then $K\in \mathcal{G}(r,L,C)$ and also, up to a subsequence, $\bdry(K_n) \to H$ in the Hausdorff distance, with $H\subset K$ compact set such that $\bdry(K) \subset H$.
\end{proposition} 

\begin{proof}
	To prove that $K\in \mathcal{G}(r,L,C)$, we have to show that for any $x\in K$ there exists $\Phi_x : B_r(x) \subset \R^N \to \R^N$ bi-Lipschitz map which satisfies conditions (a), (b) and (c$'$).  
	Let $x\in K$, then, since $K_n\to K$ in the Hausdorff distance, there exists $x_n\in K_n$ such that $x_n\to x$. Now $K_n\in \mathcal{G}(r,L,C)$ for every $n\in\N$, therefore there exists a bi-Lipschitz map $\Phi_{x_n}^n: B_r(x_n) \to \R^N$ satisfying
	(a), (b) and (c$'$). For any $y\in B_{r/2}(x_n) \cap K_n$, we call $C_y^n$ the corresponding cone in property (c$'$).
	
	We know that the maps $\tilde{\Phi}_n \colon \overline{B_r(0)} \to \R^N$ given by
		$\Phi_{x_n}^n = \tilde{\Phi}_n(\cdot - x_n)$, up to a subsequence, converge uniformly to $\tilde{\Phi} \colon \overline{B_r(0)} \to \R^N$
		and that the map $\Phi_x := \tilde{\Phi}(\cdot - x)$ satisfies (a) and (b).
		
 	To show (c$'$), let $y\in K\cap B_{r/2}(x)$, then, as before, there exist $y_n \in K_n$ such that $y_n \to y$. Up to a subsequence, $y_n \in B_{r/2}(x_n)\cap K_n$ and the corresponding cones of property (c$'$), $C_{y_n}^n $, converge in the Hausdorff distance to a cone $\tilde{C}\cong C$ such that $\Phi_x(y)\in \tilde{C}$, by uniform convergence. We need to show that $\tilde{C}=C_y$, that is, that
	$\tilde{C} \subset \Phi_x(\overline{B_r(x)}\cap K) $. Let $w\in \tilde{C}$. By Hausdorff convergence, always up to subsequences,
	we can find $z_n\in \overline{B_r(x_n)}\cap K_n$ such that $w_n=\Phi_{x_n}(z_n)\to w$. Without loss of generality, we can also assume that $z_n$ is converging to $z\in \overline{B_r(x)}\cap K$ and that, by uniform convergence, $\Phi_x(z)=w$, thus property (c$'$) is proved.
	
	We know that, up to a subsequence, $\bdry(K_n)\to H\in \mathcal{K}(\overline{\Omega})$ with $H\subset K$. By contradiction, assume that there exists $x\in\bdry(K)\setminus H$, hence there exist $c$, $0<c<r/2$, and $\{n_k\}_{k\in\N}$ such that $\dist(x,\bdry(K_{n_k})) > c$ for any $k\in \N$.
	However, by Hausdorff convergence, there exist $y_n\in K_n$ such that $y_n\to x$, so, for some $\bar{k}\in \N$, we have $\dist(y_{n_k},\bdry(K_{n_k})) > c/2$ for any $k\ge \bar{k}$. 	 
Clearly, $\Phi_{y_{n_k}}^{-1}(B_{c/(2L)})\subset
	B_{c/2}(y_{n_k})$ and, by Remark~\ref{bdrychar}, we have that $B_{c/(2L)}\cap \pi\subset \Phi_{y_{n_k}}(B_r(y_{n_k})\cap K)$. Since $x\in \bdry(K)$, there exists $z\in B_{c/(4L)}\cap \pi$ not belonging to $\Phi_x(B_r(x)\cap K)$. However, $w_k=\Phi^{-1}_{y_{n_k}}(z)\in K_{n_k}$ and $w_k\to w\in K$ and $\Phi_x(w)=z$, thus we obtain a contradiction.
\end{proof}

\begin{example}
\label{pacman}
	In the previous Proposition we showed that $\bdry(K_n) \to H$ in the Hausdorff distance with $\bdry(K) \subset H$, but it could happen that the inclusion $\bdry(K) \subset H$ is strict, indeed consider the following example.
	
	Using polar coordinates in $\R^2$, define:
	\[ A = \overline{B_1}; \quad A_n = \{ (r,\vartheta) \mid 0 \le r \le 1 , \, \frac{1}{n} \le \vartheta \le 2\pi \}	\quad \forall n\ge 1 \]
	Set $K=A \times \{0\}$ and $K_n = A_n \times \{0\}$, for every $n\ge 1$, as subsets of $\R^3$. It is easy to show that, for some $r$, $L$ and $C$, we have $K_n \in \mathcal{G}(r,L,C)$ for any $n\ge 1$ and $K \in \mathcal{G}(r,L,C)$. Moreover, $K_n \to K$ in the Hausdorff distance but
	$\partial K_n\to H$ in the Hausdorff distance, where
	\[ \bdry(K) = \partial A \times \{0\}\subsetneq H = \bdry(K) \cup (\{(r,\vartheta) \mid 0\le r \le 1, \vartheta = 0 \} \times \{0\}).  \] 
\end{example}
	
We finally show that our class $\mathcal{FR}$ generalises the class $\mathcal{G}$.

\begin{theorem}
\label{inclusione:giac}
	Let $r>0$, $L>0$ and $C\subset \R^{N-1}$ finite closed cone with vertex in $0$. Then there exist $r'>0$, $L'>0$ and $M_0\in\N$, depending only on $r$, $L$, $C$ and, only for $M_0$, also on the diameter of $\Omega$, such that for any $K\in\mathcal{G}(r,L,C)$, we have that $K\in\mathcal{FR}(r',L',M_0)$.
\end{theorem}

\begin{proof}
	Let $K\in\mathcal{G}(r,L,C)$ and observe that, without loss of generality, up to shrinking the cone $C$, we can always suppose that
	\begin{equation}
	\label{dimgiac:eqn1}
		\diam(C) < \frac{r}{8L}.
	\end{equation}
	Since $K\in\mathcal{G}(r,L,C)$, it is clear that $K$ satisfies hypotheses (a) and (b) globally, therefore to show that $K\in\mathcal{FR}(r,L,M_0)$, we just need to find a decomposition of $K$ in hypersurfaces of type $\mathcal{MR}(r',L')$.  
	Given that $K$ is compact, there exists a finite number $m\in\N$ of points $x_1,\ldots,x_m\in K$ such that $K \subset \bigcup_{i=1}^m B_{r/3}(x_i)$.
	We point out that, for some $m_0$ depending only on $r$ and the diameter of $\Omega$, we can assume that $m\leq m_0$.
		Now define for every $i=1,\ldots,m$
	\begin{equation}
	\label{dimgiac:eqn2}
		K_i := \Phi_{x_i}^{-1} \left( \overline{\bigcup_{C'\in\,\mathcal{C}_{x_i}} C'} \right) 
	\end{equation}
	where
	\begin{align*}
	\mathcal{C}_{x_i} := \{ & C' \subset \pi \text{ finite closed cone} \mid C' \cong C \text{, } C' \subset \Phi_{x_i}(\overline{B_r(x_i)} \cap K) \\ 
	& \text{and } C' \cap \Phi_{x_i}(\overline{B_{r/3}(x_i)} \cap K) \ne \emptyset \},
	\end{align*} which is not empty thanks to condition (c$'$). We have that, for every $i=1,\ldots,m$, $K_i$ is compact and, by (c$'$) and \eqref{dimgiac:eqn1}, $B_{r/3}(x_i) \subset K_i \subset B_{r/2} (x_i)$, hence $K=\bigcup_{i=1}^m K_i$.
	Moreover, by Lemma~ \ref{unione:coni}, the set $\Phi_{x_i}(K_i) \subset \pi$ satisfies the cone condition with respect to $C$. By applying Theorem~\ref{gagliardo} with $\rho = \bar{\rho}/2$, which depends only on $C$, to $\Phi_{x_i}(K_i)$ we have
	\begin{equation*}
		\Phi_{x_i}(K_i) := \bigcup_{k=1}^l A_{i,k}
	\end{equation*}
	with $A_{i,k}$ given by the closure, in $\pi$, of open subsets in $\pi$ which are Lipschitz with constants $\eta$ and $L_1$, which also depends only on the cone $C$. Note that $l\in \N$ depends only on $C$, $r$ and $L$. We call
		\begin{equation*}
		K_{i,k} := \Phi_{x_i}^{-1} (A_{i,k})
	\end{equation*}
	for every $i=1,\ldots,m$ and $k=1,\ldots,l$.
	
	It is not difficult to show that we can find positive $r'$ and $L'$, depending only on $r$, $L$, $C$, such that
	$K_{i,k}\in\mathcal{MR}(r',L')$ for every $i=1,\ldots,m$ and $k=1,\ldots,l$. If we take $M_0:=m_0l$, since $ml\leq M_0$, the proof is concluded.
	\end{proof}

\end{document}